\documentclass[10pt,reqno]{amsart}
\usepackage{color}

\newtheorem{theorem}{Theorem}[section]
\newtheorem{lemma}[theorem]{Lemma}
\newtheorem{corollary}[theorem]{Corollary}

\theoremstyle{remark}
\newtheorem{remark}[theorem]{Remark}

\theoremstyle{definition}
\newtheorem{assumption}[theorem]{Assumption}

\newtheorem{definition}[theorem]{Definition}

\newcommand\cbrk{\text{$]$\kern-.15em$]$}}
\newcommand\opar{\text{\,\raise.2ex\hbox{${\scriptstyle|}$}\kern-.34em$($}}
\newcommand\cpar{\text{$)$\kern-.34em\raise.2ex\hbox{${\scriptstyle |}$}}\,}

\def\qed{{\hfill $\Box$ \bigskip}}
\def\loc{{\rm loc}}

\def\XXint#1#2#3{{\setbox0=\hbox{$#1{#2#3}{\int}$}
\vcenter{\hbox{$#2#3$}}\kern-.5\wd0}}

\newcommand\bL{\mathbb{L}}
\newcommand\bH{\mathbb{H}}
\newcommand\bZ{\mathbb{Z}}

\newcommand\bD{\mathbb{D}}
\newcommand\bS{\mathbb{S}}
\newcommand\bE{\mathbb{E}}

\newcommand\bN{\mathbb{N}}

\newcommand\fR{\mathbf{R}}

\newcommand\fF{\mathbf{F}}

\newcommand\cB{\mathcal{B}}

\newcommand\cF{\mathcal{F}}
\newcommand\cH{\mathcal{H}}

\newcommand\cP{\mathcal{P}}
\newcommand\cR{\mathcal{R}}
\newcommand\cS{\mathcal{S}}
\newcommand\cM{\mathcal{M}}

\newcommand\cO{\mathcal{O}}

\newcommand\frH{\mathfrak{H}}

\newcommand{\mysection}[1]{\section{#1}
\setcounter{equation}{0}}

\begin{document}

\title[Non-linear SPDE$\mathbf{s}$ of divergence type]
{A regularity theory for    quasi-linear  Stochastic Partial Differential Equations in weighted Sobolev spaces}

\author{Ildoo Kim}
\address{Center for Mathematical Challenges, Korea Institute for Advanced Study (KIAS), 85 Hoegiro Dongdaemun-gu,
Seoul 130-722, Republic of Korea} \email{waldoo@kias.re.kr}
\thanks{The research of the first author was supported by the TJ Park Science  
         Fellowship of POSCO TJ Park Foundation}

\author{Kyeong-Hun Kim}
\address{Department of mathematics, Korea university, 1 anam-dong
sungbuk-gu, Seoul, south Korea 136-701}
\email{kyeonghun@korea.ac.kr}

\thanks{The research of the second
author was supported by Basic Science Research Program through the
National Research Foundation of Korea(NRF) funded by the Ministry of
Education, Science and Technology (20120005158)}

\subjclass[2000]{60H15, 35R60}

\keywords{Quasilinear stochastic partial differential equations, Equations of divergence type,  Weighted Sobolev space}

\begin{abstract}
We study  the  second-order  quasi-linear  stochastic partial differential equations (SPDEs) defined on $C^1$ domains. The coefficients are random functions  depending on $t,x$ and the unknown solutions.  We prove the uniqueness and existence of solutions in appropriate Sobolev spaces, and in addition, we obtain  $L_p$ and H\"older estimates of both the solution and its gradient.
\end{abstract}

\maketitle

\mysection{introduction}
In this article we present a weighted Sobolev space theory of the following  stochastic partial differential equation (SPDE):
\begin{align}
 \notag
du &= \Big[D_i \Big(a^{ij}(t,x,u)u_{x^j}+b^i(t,x,u) u + f^i\Big) + \bar{b}^i(t,x,u) u_{x^i} +c(t,x,u)u +f\Big]dt\\
&\quad +(\nu^k(t,x)u+g^k) dW_t^k, \qquad  t\leq \tau, \quad x\in \cO; \quad u(0,\cdot)=u_0.
					\label{main eqn}
\end{align}
Here $\tau$ is an arbitrary bounded stopping time, $\cO$ is a bounded $C^1$-domain in $\fR^d$, and $W^k_t$ $(k=1,2,\cdots)$ are independent one-dimensional Wiener processes defined on a probability space $(\Omega,\cF,P)$. The indices $i$ and $j$ move from $1$ to $d$, and $k$ runs through $\{1,2,3,\cdots\}$. The Einstein's summation convention with respect to $i,j$ and $k$ is assumed throughout the article. All the coefficients are random, the coefficients   $a^{ij}, b^i, \bar{b}^i, c$  depend  also on $t,x$ and the unknown $u$, and the coefficients $\nu^k$ ($k=1,2,\cdots$) depend  on $\omega$,  $t$, and $x$. 
We assume that  the coefficients are only measurable with respect to $(\omega,t)$, H\"older continuous with respect to $x$,  and Lipschitz continuous with respect to the unknown $u$.

Let $u(t,x)$ denote the density of diffusing particles at the time $t$ and the location $x$. Typically, the flux density $\fF(t,x)$ is proportional to $-\nabla u$ or more generally to $-\sum_j  a^{ij}u_{x^j}$, and the classical heat equation $u_t=D_i(a^{ij}u_{x^j})$ is a consequence of  the relation $u_t=-\text{div} \,\fF$. 
Then motivation of studying equation \eqref{main eqn} is obvious since the diffusion coefficients $a^{ij}$ related to the flux density $\fF(t,x)$ can  depend also on their point density $u(t,x)$. Our equation is this type general equation with noises and random external forces. 
The external forces $f^i$, $f$, and $g$ are contained in a weighted Sobolev space.
More precisely,
$$
f^{i}\in \bH^{\gamma_0}_{p,d}(\cO,\tau):=L_p(\Omega\times (0,\tau];H^{\gamma_0}_{p,d}(\cO)),
~f\in  \bH^{\gamma_0-1}_{p,d+p}(\cO,\tau),
~g \in\bH^{\gamma_0}_{p,d}(\cO,\tau,\l_2),
$$
where $p > d+2$ and $\gamma_0 \in ((d+2)/p,1)$. 
The spaces $H^{\gamma}_{p,d}(\cO)$ and $\bH^{\gamma}_{p,d}(\cO,\tau)$ are  introduced in Section 2. We only remark that if $\gamma$ is a natural number then $u \in H^{\gamma}_{p,d}(\cO)$
iff
$$
\rho^{|\alpha|}D^\alpha u \in L_p(\cO), \quad \forall |\alpha| \leq \gamma \quad \quad (\rho(x):=\text{dist}(x, \partial \cO)).
$$
 Under the setting, we prove that equation \eqref{main eqn} has a unique solution  in a certain weighted Sobolev space, actually in $\frH^1_{2,d}(\cO,\tau)\cap \frH^{1+\gamma_0-\varepsilon}_{p,d,\loc}(\cO,\tau)$,  ($\varepsilon>0$).  
Using    an embedding theorem related to this weighted Sobolev space, we get  the following interior H\"older estimates of
 the solution $u$ and its gradient $\nabla u$: 
for any constants $\kappa,\kappa_1,\alpha,\beta,\beta_1$ satisfying
  \begin{eqnarray*}
  1/p<\kappa_1<\kappa <1/2,  \quad 1-2\kappa-d/p>0,  \\
  \alpha<\gamma_0, \quad 1/p<\beta_1<\beta<1/2, \quad \alpha-2\beta-d/p>0,
  \end{eqnarray*}
it holds that for all $t<\tau$ (a.s.),
\begin{eqnarray}
|\rho^{-\varepsilon}u|_{C^{\kappa_1-1/p}([0,t],C^{1-2\kappa-d/p-\varepsilon}(\cO))}  <\infty,    \quad \forall \varepsilon \in [0, 1-2\kappa-d/p]    \label{good1}
\end{eqnarray}
  \begin{equation}
                                                 \label{good}
   |\rho^{\alpha-\varepsilon'}u_x|_{C^{\beta_1-1/p}([0,t],C^{\alpha-2\beta-d/p-\varepsilon'}(\cO))}\quad <\,\infty, \quad \forall \varepsilon' \in [0,\alpha-2\beta-d/p].
  \end{equation}
  Note that  (\ref{good1})  is a much better result than the one without $\rho$. For instance, (\ref{good1}) with  $\varepsilon=1-2\kappa-d/p$  implies
  $\sup_{s\leq t}|u(s,x)|=O(\rho^{1-d/p-2\kappa}(x))$. This shows how fast $u(t,x)\to 0$ as $\rho(x)\to 0$.
  By taking $p\to \infty$ one can make  $1-2\kappa-d/p$ (the H\"older exponent  of $u$ in  space variable) as close to $1$ as one wishes, and similarly
   $\kappa_1-1/p$ (the H\"older exponent of $u$ in time variable)  can be any number  in $(0,1/2)$.
  Also note that (\ref{good}) implies   that $u_x$ is H\"older continuous in $(t,x)$ only interior of $\cO$, that is $|u_x|^p_{C^{\beta_1-1/p}([0,t],C^{\alpha-2\beta-d/p}(K))}<\infty$ for  any compact set $K\subset \cO$.


Below we introduce  some related results handling divergence or non-divergence type SPDEs whose leading coefficients depend also on the solution $u$. 
 The  $1$-dimensional non-divergence type equation
\begin{align*}
du = a(t,x,u) u''dt +(b(t,x)u'+h(t,x)u) dW_t, \qquad  t\leq \tau, \quad x\in \fR.
\end{align*}
is studied \cite{da1996fully} under the assumption that coefficients $a$, $b$, $h$ are infinitely differentiable with bounded  derivatives.  A similar   equation
\begin{equation}
du=[a(t,x,u)u''+f(t,x)]dt+ g^k(t,x)  dW^k_t.
\end{equation}
 is studied in \cite{yoo1998unique}. Compared to  \cite{da1996fully},  the condition on $a$ is much weaker in \cite{yoo1998unique}. Here the diffusion coefficient $a(t,x,u)$ is H\"older continuous in $t$, differentiable in $x$, and twice continuously differentiable in $u$. Howerver both \cite{da1996fully} and \cite{yoo1998unique} considered only one-dimensional equation.
In \cite{kim2014some},  we obatined  $L_p$ and H\"older estimates  for the divergence type equation
\begin{align*}
du &= \Big[D_i \Big(a^{ij}(t,x,u)u_{x^j}+ f^i\Big) + f\Big]dt +g^k dW_t^k, \qquad  t\leq \tau, \quad x\in \cO,
\end{align*}
where $a^{ij}(t,x,u)$ are H\"older continuous in $x$ and twice continuously differentiable in $u$.
The present article is a generalization of \cite{kim2014some}. Firstly, we generalize the equation. We have multiplicative noises in the stochastic part of
 \eqref{main eqn} together with nonlinear lower order terms of solutions in the deterministic part.  Secondly, our smoothness conditions on the coefficients are  weaker than those in  \cite{kim2014some}. We only impose the Lipschitz continuity  to $a^{ij}(t,x,u)$ with respect to $u$.  Thirdly,  our $L_p$ and H\"older theory work for any $p >d+2$ and $\gamma_0 \in ((d+2)/p,1)$, whereas in \cite{kim2014some}, $p$ and $\gamma_0$ are some constants (hard to know exactly) coming from the deterministic theory.

Our approach is based on a weighted Sobolev space theory for divergence type linear SPDEs.  It might be possible to study equation 
(\ref{main eqn}) using an infinite dimensional SDE theory so called variational approach.
See, for instance, \cite{krylov1981stochastic, liustochastic} and references therein. The monotonicity or local monotonicity condition is crucial in such theory.  It is easy to check  that the operator $D_i(a^{ij}(t,x,u)u_{x^j})$ does not satisfy the monotonicity condition, but it is not clear to us if local monotonicity condition holds for this operator. Regardless of the possibility of using the variational approach, our approach has many advantages. In particular, it provides $L_p$ and H\"older estimates of both the solution and its gradient. Furthermore, as can be seen in (\ref{good1}) and (\ref{good}), it provides very delicate behaviors of the solution and its derivatives near the boundary.

This paper is organized as follows. We introduce our main results and related function spaces in Section 2. 
In Section 3, we collect some auxiliary results related to linear SPDEs. 
The (time) local well-poseness of equation \eqref{main eqn} is given in Section 4.
Finally, the proof of the main theorem is presented in Section 5.

We finish the introduction with  notation used in the article. 
$\bN$ and $\bZ$ denote the natural number system and the integer number system, respectively.
As usual $\fR^{d}$
stands for the Euclidean space of points $x=(x^{1},...,x^{d})$,
$\fR^d_+:=\{x=(x^1,\cdots,x^d)\in \fR^d: x^1>0\}$ and
$B_r(x):=\{y\in \fR^d: |x-y|<r\}$.
 For $i=1,...,d$, multi-indices $\alpha=(\alpha_{1},...,\alpha_{d})$,
$\alpha_{i}\in\{0,1,2,...\}$, and functions $u(x)$ we set
$$
u_{x^{i}}=\frac{\partial u}{\partial x^{i}}=D_{i}u,\quad
D^{\alpha}u=D_{1}^{\alpha_{1}}\cdot...\cdot D^{\alpha_{d}}_{d}u,
\quad  \nabla u=(u_{x^1}, u_{x^2}, \cdots, u_{x^d}).
$$
We also use the notation $D^m$ for a partial derivative of order $m$ with respect to $x$.  
For $p \in [1,\infty)$, a normed space $F$ 
and a  measure space $(X,\mathcal{M},\mu)$, $L_{p}(X,\cM,\mu;F)$
denotes the space of all $F$-valued $\mathcal{M}^{\mu}$-measurable functions
$u$ so that
\[
\left\Vert u\right\Vert _{L_{p}(X,\cM,\mu;F)}:=\left(\int_{X}\left\Vert u(x)\right\Vert _{F}^{p}\mu(dx)\right)^{1/p}<\infty,
\]
where $\mathcal{M}^{\mu}$ denotes the completion of $\cM$ with respect to the measure $\mu$. 
For $p=\infty$, we write $u \in L_{\infty}(X,\cM,\mu;F)$ iff
$$
\sup_{x}|u(x)| := \|u\|_{L_{\infty}(X,\cM,\mu;F)} 
:= \inf\left\{ \nu \geq 0 : \mu( \{ x: \|u(x)\|_F > \nu\})=0\right\} <\infty.
$$
If there is no confusion for the given measure and $\sigma$-algebra, we usually omit the measure and the $\sigma$-algebra. 
If we write $N=N(a,b,\cdots)$, this means that the
constant $N$ depends only on $a,b,\cdots$. 
 For functions  depending on $\omega$, $t$,  and $x$, the argument
$\omega \in \Omega$ will be usually omitted. 
We say that a stopping time $\tau$ is nonzero iff $P(\{ \tau \neq 0 \}) > 0$.

\mysection{main result}

Let $(\Omega,\cF,P)$ be a probability space and let
$\{\cF_{t},t\geq0\}$ be  an increasing filtration of
$\sigma$-fields on $\Omega$ satisfying the usual condition, i.e. 
$\cF_{t}\subset\cF$ contains all $(\cF,P)$-null sets and $\cF_t = \bigcap_{s>t} \cF_s$.
 By $\cP$  we denote the predictable $\sigma$-field, that is  $\cP$ is the smallest $\sigma$-field containing
the collection of all sets $A \times (s,t]$, where $0 \leq s \leq t <\infty$ and $A \in \cF_s$. 
The processes $W^1_t, W^2_t,\cdots$ are  independent one-dimensional Wiener
processes defined on $\Omega$, each of which is a Wiener
process relative to $\{\cF_{t},t\geq0\}$.


For $p >1$ and $\gamma \in \fR$, let
$H_{p}^{\gamma}=H_{p}^{\gamma}(\fR^{d})$ denote  the class of all
(tempered) distributions $u$  on $\fR^{d}$ such that
\begin{equation}
        \label{eqn norm}
\| u\| _{H_{p}^{\gamma}}:=\|(1-\Delta)^{\gamma/2}u\|_{L_{p}}<\infty,
\end{equation}
where
$$
(1-\Delta)^{\gamma/2} u = \cF^{-1} \left((1+|\xi|^2)^{\gamma/2}\cF (u) \right).
$$
Here $\cF$ and $\cF^{-1}$ are  Fourier and inverse Fourier transforms respectively. 
It is well-known that if $\gamma=1,2,\cdots$, then
$$
H^{\gamma}_p=W^{\gamma}_p:=\{u: D^{\alpha}_x u\in L_p(\fR^d), \, \,\,|\alpha|\leq \gamma\}, \quad \quad H^{-\gamma}_p=\left(H^{\gamma}_{p/{(p-1)}}\right)^*,
$$
where $\left(H^{\gamma}_{p/{(p-1}}\right)^*$ denotes the dual space of $H^{\gamma}_{p/{(p-1)}}$.
For  a  tempered distribution $u\in H^{\gamma}_p$ and $\phi\in
\cS(\fR^d)$, the action of $u$ on $\phi$ (or the image of $\phi$
under $u$) is defined as
$$(u,\phi)=\left((1-\Delta)^{\gamma/2}u ,
(1-\Delta)^{-\gamma/2}\phi \right)=\int_{\fR^d}
(1-\Delta)^{\gamma/2}u(x) \cdot (1-\Delta)^{-\gamma/2}\phi(x) \,dx.
$$
Let  $l_2$ denote the set of all sequences $a=(a^1,a^2,\cdots)$ such that
$$|a|_{l_{2}}:= \left(\sum_{k=1}^{\infty}|a^{k}|^{2}\right)^{1/2}<\infty.
$$
By $H_{p}^{\gamma}(l_{2})=H_{p}^{\gamma}(\fR^d ; l_2)$  we denote the class of all $l_2$-valued
(tempered) distributions $v=(v^1,v^2,\cdots)$ on $\fR^{d}$ such that
$$
\|v\|_{H_{p}^{\gamma}(l_{2})}:=\||(1-\Delta)^{\gamma/2}v|_{l_2}\|_{L_{p}}<\infty.
$$

\vspace{2mm}

Next we introduce weighted Sobolev spaces $H^{\gamma}_{p,\theta}(\cO)$ defined on domains, where $\gamma,\theta\in \fR$.  
Let $\cO$ be a bounded $C^1$ domain in $\fR^d$ and denote $\rho(x):=\text{dist}(x,\partial \cO)$. 
Then one can choose a smooth function $\psi$ defined on $\bar{\cO}$ satisfying the followings (see, e.g.
\cite{gilbarg1980intermediate, kim2004sobolev}):

$\bullet $ \, $\psi $ is comparable to $\rho$, that is there is a constant $N=N(\cO)$ so that 
 $$
N^{-1}\rho(x) \leq
\psi(x) \leq N\rho(x),\quad \forall x\in \cO.
$$

 $\bullet$ \,  $\psi$ is infinitely differentiable in $\cO$ (not up to the boundary), and for any  multi-index $\alpha$,
\begin{equation}
                                                             \label{03.04.01}
\sup_{\cO} \psi ^{|\alpha|}(x)|D^{\alpha}\psi_{x}(x)| <\infty. \nonumber
\end{equation}

 Fix a nonnegative  function $\zeta(x)=\zeta(x^1)\in
C^{\infty}_0(\fR_+)$ such that
\begin{equation}
                                       \label{eqn 5.6.5}
\sum_{n=-\infty}^{\infty}\zeta^p(e^{n}x^1)>c>0, \quad \forall x^1\in \fR_+,
\end{equation}
where $c$ is a constant.  Note that    any nonnegative function $\zeta$
with $\zeta>0$ on $[1,e]$ satisfies (\ref{eqn 5.6.5}).
 For $x\in \cO$ and $n\in\bZ=\{0,\pm1,...\}$
define
$$
\zeta_{n}(x)=\zeta(e^{n}\psi(x)).
$$
Then  we have $\sum_{n}\zeta_{n}\geq c>0$ in $\cO$ and
\begin{equation*}
\zeta_n \in C^{\infty}_0(\cO), \quad |D^m \zeta_n(x)|\leq
N(m)e^{mn}.
\end{equation*}
For $\theta,\gamma \in \fR$, let $H^{\gamma}_{p,\theta}(\cO)$ be the
set of all distributions $u$  on $\cO$ such
that
\begin{equation}
                                                 \label{10.10.03}
\|u\|_{H^{\gamma}_{p,\theta}(\cO)}^{p}:= \sum_{n\in\bZ} e^{n\theta}
\|\zeta_{-n}(e^{n} \cdot)u(e^{n} \cdot)\|^p_{H^{\gamma}_p} < \infty. \nonumber
\end{equation}
Similarly, for  $\l_2$-valued functions $g=(g^1,g^2,\cdots)$ we define
$$
\|g\|^p_{H^{\gamma}_{p,\theta}(\cO,\l_2)}=\sum_{n\in\bZ} e^{n\theta}
\|\zeta_{-n}(e^{n} \cdot)g(e^{n} \cdot)\|^p_{H^{\gamma}_p(\l_2)}.
$$

\noindent
It is known (see \cite{krylov1999weighted,lototsky2000sobolev}) that up to equivalent
norms the space $H^{\gamma}_{p,d}(\cO)$ is independent of the
choice of $\zeta$ and $\psi$. Moreover if $\gamma$ is a
non-negative integer then
$$
H^{\gamma}_{p,\theta}=\{u: \rho^{|\alpha|}D^{\alpha}u\in L_p(\cO, \rho^{\theta-d}dx),  |\alpha|\leq \gamma\},
$$
and
\begin{equation}
                              \label{eqn 02.09.1}
\|u\|^p_{H^{\gamma}_{p,\theta}(\cO)} \sim
\sum_{|\alpha|\leq \gamma}\int_{\cO} |\rho^{|\alpha|}D^{\alpha}u(x)|^p
\rho^{\theta-d}(x) \,dx.  \nonumber
\end{equation}

To   state our assumptions on the coefficients, we take some
notation from \cite{gilbarg2015elliptic,kim2004sobolev}. Denote $\rho(x,y)=\rho(x)\wedge \rho(y)$. For
   $\alpha\in(0,1]$,
 and $k=0,1,2,...$,  we define interior H\"older norm $|\cdot|^{(0)}_{k+\alpha}$ as follows.
 $$
 |f|^{(0)}_k=\sum_{|\beta|\leq k} \sup_{\cO} \rho^{|\beta|}(x)|D^{\beta}f(x)|,
 $$
 $$
  [f]^{(0)}_{k+\alpha}
=\sup_{\substack{x,y\in \cO\\ |\beta|=k}}
\rho^{k+\alpha}(x,y)\frac{|D^{\beta}f(x)-D^{\beta}f(y)|}
{|x-y|^{\alpha}}
 $$
$$
|f|^{(0)}_{k+\alpha}=|f|^{(0)}_{k}+
[f]^{(0)}_{k+\alpha}.
$$
For $l_2$-valued functions $f=(f^1,f^2,\cdots)$ we define $|f|^{(0)}_{k+\alpha}$  by using $|D^{\beta}f(x)|_{l_2}$ and $|D^{\beta}f(x)-D^{\beta}f(y)|_{l_2}$ in place of
$|D^{\beta}f(x)|$ and $|D^{\beta}f(x)-D^{\beta}f(y)|$, respectively.
One can easily check that there exists a constant $N>0$ such that for any $\gamma \in [0,1]$,
\begin{align}
					\label{0222 eqn 1}
|f|^{(0)}_{\gamma}
\leq N \left(|f|_{C(\cO)}+|\psi D f|_{C(\cO)} \right),
\end{align}
where $N$ is independent of $\gamma$ and $f$.

Below we collect some well-known properties of the space $H^{\gamma}_{p,\theta}(\cO)$.  For $\alpha \in \fR$, we write $f\in \psi^{\alpha}H^{\gamma}_{p,\theta}(\cO)$ if and only if $\psi^{-\alpha}f\in H^{\gamma}_{p,\theta}(\cO)$.

\begin{lemma}
                        \label{wei lem}

 (i) For any $\gamma,\theta\in \fR$, $C^{\infty}_c(\cO)$ is dense in $H^{\gamma}_{p,\theta}(\cO)$.

(ii) Assume that $\gamma-d/p=m+\nu$ for some $m=0,1,\cdots$ and
$\nu\in (0,1]$.  Then for any $u\in H^{\gamma}_{p,\theta}(\cO)$ and $i\in
\{0,1,\cdots,m\}$, we have
\begin{align*}
|\psi^{i+\theta/p}D^iu|_{C(\cO)}+|\psi^{m+\nu+\theta/p}D^m
u|_{C^{\nu}}\leq N \|u\|_{ H^{\gamma}_{p,\theta}(\cO)}.
\end{align*}

(iii) Let $\alpha\in \fR$, then
$\psi^{\alpha}H^{\gamma}_{p,\theta+\alpha
p}(\cO)=H^{\gamma}_{p,\theta}(\cO)$,
\begin{align*}
\|u\|_{H^{\gamma}_{p,\theta}(\cO)}\leq N
\|\psi^{-\alpha}u\|_{H^{\gamma}_{p,\theta+\alpha p}(\cO)}\leq
N\|u\|_{H^{\gamma}_{p,\theta}(\cO)}.
\end{align*}

(iv) Let $\varepsilon=0$ if $\gamma$ is an integer, and $\varepsilon>0$ otherwise. Then
\begin{align*}
\|a u\|_{H^{\gamma}_{p,\theta}(\cO)}\leq
N |a|^{(0)}_{|\gamma|+\varepsilon}\|u\|_{H^{\gamma}_{p,\theta}(\cO)}.
\end{align*}

(v) $\psi D_i, D_i\psi: H^{\gamma}_{p,\theta}(\cO)\to
H^{\gamma-1}_{p,\theta}(\cO)$ are bounded linear operators, and
$$
\|u\|_{H^{\gamma}_{p,\theta}(\cO)}\leq
N\|u\|_{H^{\gamma-1}_{p,\theta}(\cO)}+N \|\psi
\nabla u\|_{H^{\gamma-1}_{p,\theta}(\cO)}\leq N
\|u\|_{H^{\gamma}_{p,\theta}(\cO)},
$$
$$
\|u\|_{H^{\gamma}_{p,\theta}(\cO)}\leq
N\|u\|_{H^{\gamma-1}_{p,\theta}(\cO)}+N \|\nabla (\psi
u)\|_{H^{\gamma-1}_{p,\theta}(\cO)}\leq N
\|u\|_{H^{\gamma}_{p,\theta}(\cO)}.
$$

\end{lemma}

Now we introduce stochastic Banach spaces. 
For a stopping time $\tau$, denote $\opar 0, \tau \cbrk=\{(\omega,t): 0<t\leq \tau(\omega)\}$,
$$
\bH^{\gamma}_{p,\theta}(\cO,\tau)=L_p(\opar 0,\tau \cbrk,\cP;H^{\gamma}_{p,\theta}(\cO)), \quad
\bH^{\gamma}_{p,\theta}(\cO,\tau,l_2)=L_p(\opar 0,\tau \cbrk,\cP;H^{\gamma}_{p,\theta}(\cO,l_2)),
$$
$$
\bL_{p,\theta}(\cO,\tau)=\bH^{0}_{p,\theta}(\cO,\tau),\quad 
U^{\gamma}_{p,\theta}(\cO)=\psi^{1-2/p}L_p(\Omega,\cF_0;H^{\gamma-2/p}_{p,\theta}(\cO)),
$$
where
$$
\|u\|^p_{\bH^{\gamma}_{p,\theta}(\cO,\tau)}=\bE\int^{\tau}_0 \|u\|^p_{H^{\gamma}_{p,\theta}(\cO)}dt, \quad
\|f\|^p_{U^{\gamma}_{p,\theta}(\cO)}:=\bE\|\psi^{2/p-1}f\|^p_{H^{\gamma-2/p}_{p,\theta}(\cO)}.
$$
For instance,  $u\in \bH^{\gamma}_{p,\theta}(\cO,\tau)$ if $u$  has an  $H^{\gamma}_{p,\theta}(\cO)$-valued $\cP$-measurable version   $v$ defined on $\opar 0, \tau \cbrk$ (i.e. $u=v$ a.e. in $\opar 0, \tau \cbrk$) and  $\|u\|_{\bH^{\gamma}_{p,\theta}(\cO,\tau)}<\infty$.

\begin{definition}
                \label{definition name}
We write   $u\in \frH^{\gamma+1}_{p,\theta}(\cO,\tau)$  if
 $u\in \psi\bH^{\gamma+1}_{p,\theta}(\cO,\tau)$,
$u_0 \in U^{\gamma+1}_{p,\theta}(\cO)$ and  for some $h \in
\psi^{-1}\bH^{\gamma-1}_{p,\theta}(\cO,\tau)$ and $g=(g^1,g^2,\cdots)\in \bH^{\gamma}_{p,d}(\cO,\tau,\l_2)$,  it holds that  
$$
du=hdt+g^kdw^k_t, \qquad t\leq \tau; \quad u(0,\cdot)=u_0
$$
in the sense of distributions, that is for any $\phi\in C^{\infty}_c(\cO)$, the equality
$$
(u(t),\phi)=(u_0,\phi)+\int^t_0 (h(s),\phi)ds+\sum_k \int^t_0 (g^k(s),\phi)dw^k_s
$$
holds for all $t\leq \tau$ (a.s.). In this case we write
$$
 h=\bD u, \quad \text{and} \quad g=\bS u.
$$
Especially, we say that $u$ is a solution to equation \eqref{main eqn} if
\begin{align*}
\bD u = D_i(a^{ij}(u)u_{x^j}+b^i(u) u+ \bar{b}^i(u) u_{x^i} +c(u)u+ f^i) +f,  \quad
\bS u = \nu u+g.
\end{align*}
The norm in  $\frH^{\gamma+1}_{p,\theta}(\cO,\tau)$ is given by
$$
\|u\|_{\frH^{\gamma+1}_{p,\theta}(\cO,\tau)}=
\|\psi^{-1}u\|_{\bH^{\gamma+1}_{p,\theta}(\cO,\tau)} + \|\psi
\bD u\|_{\bH^{\gamma-1}_{p,\theta}(\cO,\tau)}  + \|\bS u\|_{\bH^{\gamma}_{p,\theta}(\cO,\tau,\l_2)}+
\|u(0,\cdot)\|_{U^{\gamma+1}_{p,\theta}(\cO)}.
$$
Finally,
we write $u \in \frH^{\gamma+1}_{p,\theta,\text{loc}}(\cO,\tau)$ if there exists a sequence of stopping times  $\tau_n \uparrow \tau$  so that  $u \in \frH^{\gamma+1}_{p,\theta}(\cO,\tau_n)$.
\end{definition}

\begin{theorem}
                                                     \label{embedding}
%
(i) For any $\gamma, \theta \in \fR$ and $p\geq 2$, $\frH^{\gamma+1}_{p,\theta}(\cO,\tau)$ is a Banach space. 

(ii)  If $\tau\leq T$, $p>2$ and
 $1/2 > \beta > \alpha > 1/p$, then for any  $u \in \frH^{\gamma+1}_{p,\theta}(\cO,\tau)$,  it holds  that $u \in C^{\alpha -1/p}([0,\tau], H^{\gamma+1-2\beta}_{p,\theta}(\cO))$ $(a.s.)$ and
\begin{equation}
                       \label{eqn 111}
\bE | \psi^{2\beta-1} u|^p_{C^{\alpha -1/p}([0,\tau],H_{p,\theta}^{1+\gamma-2\beta}(\cO))} \leq N(d,p, \alpha, \beta, T) \|u\|^p_{\frH^{1+\gamma}_{p,\theta}(\cO,\tau)}.
\end{equation}

(iii) If $p=2$ then (\ref{eqn 111}) holds with $\beta=1/2$ and $\alpha=1/p=1/2$. That is, if $u\in \frH^{\gamma+1}_{2,\theta}(\cO,\tau)$, then 
$u\in C([0,\tau];H^{\gamma}_{2,\theta}(\cO))$ (a.s.), and 
$$
\bE \sup_{t\leq T} \|u\|^2_{H^{\gamma}_{2,\theta}(\cO)}\leq N\|u\|^2_{\frH^{\gamma+1}_{2,\theta}(\cO,\tau)}.
$$
\end{theorem}

\begin{proof}
This theorem is proved by Krylov  in \cite{{krylov2001some}} if $\cO=\fR^d_+$; see  \cite[Theorem 4.1]{krylov2001some} for (ii)  and \cite[Remark 4.5]{{krylov2001some}} for (iii).
For bounded $C^1$ domains, see e.g. Lemma 3.1 of \cite{kim2004lq}  and (2.21) of \cite{lototsky2001linear}.
\end{proof}

In this article we mostly  use the above theorem when  $\theta=d$, and thus we 
 only consider the case $\theta =d$  in the following corollary.
\begin{corollary}
					\label{hol cor}
(i) Let $\hat{\alpha}:=1-d/p-2/p>0$. Then for any $\kappa\in (0,\hat{\alpha})$,
\begin{equation}
       \label{eqn 10.20}
       \bE |u|^p_{C([0,\tau];C^{\kappa}(\cO))}+\bE|u|^p_{C^{\kappa/2}([0,\tau] ;C(\cO))} \leq N \|u\|^p_{\cH^1_{p,d}(\cO,\tau)}.
       \end{equation}

	(ii) Let the constants  $\kappa,\kappa_1,\gamma,\beta$ and $\beta_1$ staisfy
  \begin{eqnarray*}
  1/p<\kappa_1<\kappa <1/2,  \quad 1-2\kappa-d/p>0,  \\
 1/p<\beta_1<\beta<1/2, \quad \gamma-2\beta-d/p>0.
  \end{eqnarray*}
Then for any $\varepsilon$ and $\varepsilon'$ satisfying
$$
0\leq \varepsilon \leq 1-2\kappa-d/p, \quad 0\leq \varepsilon'\leq \gamma-2\beta-d/p,
$$
we have
\begin{align}
						 \label{first}
&\bE|\psi^{-\varepsilon}u|^p_{C^{\kappa_1-1/p}([0,\tau],C^{1-2\kappa-d/p-\varepsilon}(\cO))}
+\bE|\rho^{\gamma-\varepsilon'}u_x|^p_{C^{\beta_1-1/p}([0,\tau],C^{\gamma-2\beta-d/p-\varepsilon'}(\cO))} \\
						\notag
& \leq N\|u\|^p_{\frH^{1+\gamma}_{p,d}(\cO,\tau)}. 
\end{align}
In particular,  we have
\begin{equation}
          \label{embed eqn}
       \bE |u |^p_{C([0,\tau]\times \cO))}+\bE | \psi^\gamma D u |^p_{C([0,\tau]\times \cO))}\leq N  \|u\|^p_{\frH_{p,d}^{1+\gamma}(\cO, \tau)}.
       \end{equation}
\end{corollary}

\begin{proof}
(i) We only consider the first term of (\ref{eqn 10.20}). The second one can be treated similarly.  Take $1/p<\alpha<\beta<1/2$ such that 
$$
1-2\beta-d/p =\kappa.
$$
To apply  Lemma \ref{wei lem}(ii) we take $\gamma=1-2\beta$, $\theta=d+p(2\beta-1)$ and $\nu=\gamma-d/p=\kappa$, and get
$$
|u|_{C^{\kappa}(\cO)}\leq N\|u\|_{H^{1-2\beta}_{p,d+p(2\beta-1)}(\cO)}\leq N \|\psi^{2\beta-1}u\|_{H^{1-2\beta}_{p,d}(\cO)},
$$
where the second inequality is due to Lemma \ref{wei lem}(iii). Therefore  the claim follows from (\ref{eqn 111}). 

(ii)  Since, $u\in \frH^{1}_{p,d}(\cO,\tau)$,  by  Theorem \ref{embedding}, we get
$$
\bE |\psi^{-1+2\kappa}u|^p_{C^{\kappa_1-1/p}([0,\tau],H^{1-2\kappa}_{p,d}(\cO))}
\leq N \|u\|^p_{\frH^{1}_{p,d}(\cO,\tau)}<\infty.
$$
  Lemma \ref{wei lem}(ii) with $\gamma=1-2\kappa$ and $\nu=1-2\kappa-d/p$,
$$
|\psi^{-1+d/p+2\kappa}u|_{C(O)} + |u|_{C^{1-d/p-2\kappa}(\cO)}\leq N\|\psi^{-1+2\kappa}u\|_{H^{1-2\kappa}_{p,d}(\cO)}.
$$
This takes care of the first term of (\ref{first}) if  $\varepsilon=1-d/p-2\kappa$ or $\varepsilon=0$. Also, if $\varepsilon<1-2\kappa-d/p$, 
apply  Lemma \ref{wei lem}(ii)  with $\gamma=1-2\kappa-\varepsilon$ and $\nu=1-2\kappa-d/p-\varepsilon$ to get
$$
|\psi^{-\varepsilon}u|_{C^{1-d/p-2\kappa-\varepsilon}(\cO)}\leq N\|\psi^{-1+2\kappa}u\|_{H^{1-2\kappa-\varepsilon}_{p,d}(\cO)}\leq N \|\psi^{-1+2\kappa}u\|_{H^{1-2\kappa}_{p,d}(\cO)}.
$$
Hence the first term of (\ref{first}) is handled. The second term is treated similarly using $u\in \frH^{1+\gamma}_{p,d}(\cO,\tau)$,  instead of $u\in \frH^{1}_{p,d}(\cO,\tau)$.
\end{proof}

Below are our assumptions on the coefficients.

\begin{assumption}[Measurability] 
					\label{as mea}
The coefficients $a^{ij}(t,x,u)$, $b^i(t,x,u)$, $\bar{b}^i(t,x,u)$, $c(t,x,u)$ 
are  $\cP \times \cB(\fR^d) \times \cB(\fR) $-measurable 
and $\nu^k(t,x)$  are $\cP \times \cB(\fR^d)$-measurable.  \end{assumption}

\begin{assumption}[Ellipticity and Boundedness]                                                          \label{as el}
  There exist   constants $\delta_0\in (0,1]$ and $K_1>0$ 
 such that
$$
 \delta_0 |\xi|^2 \leq {a}^{ij}(t,x,u) \xi^i \xi^j \leq   \delta^{-1}_0 |\xi|^2
$$
and
$$ 
 |a^{ij}(t,x,u)|+  |b^i(t,x,u)|+|\bar{b}^i(t,x,u)|+|c(t,x,u)| +|\nu^k(t,x)|_{l_2} \leq K_1
 $$
  for all $\omega,t,x,u$, and $\xi\in \fR^d$.
  \end{assumption}

\begin{assumption}[Interior H\"older continuity in $x$]
					   \label{as bd 2}
 $p  > d+2 $, $\gamma_0 \in \left(\frac{d+2}{p},1\right)$, and             
 there exists a $K_2$ such that for any $\omega,t,u$,
\begin{align*}
[a^{ij}(t,\cdot,u)]^{(0)}_{\gamma_0} 
+[ b^i(t,\cdot,u)]^{(0)}_{\gamma_0}
+[ \bar b^i(t,\cdot,u)]^{(0)}_{\gamma_0}
+[ c(t,\cdot,u)]^{(0)}_{\gamma_0}
+|\nu(t,\cdot)|^{(0)}_{1+\gamma_0} 
\leq K_2.
\end{align*}
 \end{assumption}

\begin{assumption}[Lipschitz continuity with respect to the unknown]
					\label{as co con}
There exists a constant $K_3$ such that for any $\omega,t,x,u,v,i,j$,
\begin{align*}
&|a^{ij}(t,x,u) - a^{ij}(t,x,v)| 
+\psi(x)|b^{i}(t,x,u) - b^{i}(t,x,v)| +\psi(x)|\bar b^i(t,x,u) - \bar b^i (t,x,v)| \\
&\quad +\psi^2(x)|c(t,x,u) - c(t,x,v)| \leq K_3|u-v|.
\end{align*}
\end{assumption}

Here is the main result of this article.

\begin{theorem}
                             \label{main thm}
Let  $\tau \leq T$ be a stopping time.
Suppose Assumptions \ref{as mea}, \ref{as el},  \ref{as bd 2}, and \ref{as co con} hold.   
Then  for any given $f^{i}\in \bH^{\gamma_0}_{p,d}(\cO,\tau)$,
$f\in \psi^{-1} \bH^{\gamma_0-1}_{p,d}(\cO,\tau)$,
$g \in\bH^{\gamma_0}_{p,d}(\cO,\tau,\l_2)$, 
and $u_0 \in U^{\gamma_0+1}_{p,d}(\cO)$, equation (\ref{main eqn}) has  a unique solution $u$ in $\frH^1_{2,d}(\cO,\tau)$, and for this solution $u$, we have
\begin{equation}
					\label{main est}
\|u\|_{\frH^{1}_{2,d}(\cO,\tau)} 
\leq
N ( \|f^{i}\|_{\bL_{2,d}(\cO,\tau)} +
\|\psi f\|_{\bH^{-1}_{2,d}(\cO,\tau)}
+ \|g\|_{\bL_{2,d}(\cO,\tau,\l_2)} +
 \|u_{0}\|_{U^{1}_{2,d}(\cO)}),
\end{equation}
where  $N$  depends only on $d$, $p$, $\delta_0$, $K_1$, $T$, and $\cO$.
Furthermore, 
$$
u\in  \frH^{1+\gamma}_{p,d,\loc}(\cO,\tau), \quad 
 \forall \,\, \gamma<\gamma_0,
 $$
  and   for any constants   $\kappa,\kappa_1,\gamma,\beta, \beta_1,\varepsilon$ and $\varepsilon'$ satisfying
  \begin{eqnarray*}
  1/p<\kappa_1<\kappa <1/2,  \quad 1-2\kappa-d/p>0,  \\
 1/p<\beta_1<\beta<1/2, \quad \gamma-2\beta-d/p>0,
  \end{eqnarray*}
$$
 \varepsilon \in [0, 1-2\kappa-d/p], \quad  \varepsilon' \in [0, \gamma-2\beta-d/p],
$$
it holds that for all $t<\tau$ (a.s.)
\begin{eqnarray}
          \label{holder main}
|\rho^{-\varepsilon}u|_{C^{\kappa_1-1/p}([0,t],C^{1-2\kappa-d/p-\varepsilon}(\cO))}+
|\rho^{\gamma-\varepsilon'}u_x|^p_{C^{\beta_1-1/p}([0,t],C^{\gamma-2\beta-d/p-\varepsilon'}(\cO))}<\infty.
\end{eqnarray}

\end{theorem}

\begin{remark}
(i) Taking $\varepsilon=0$  in (\ref{holder main}), we find that $u$ is H\"older continuous in $t$ with exponent $\kappa_1-1/p$ (which can be very close to $1/2$ if $p$ is large) and H\"older continuous in $x$ with exponent $1-2\kappa-d/p$ (this can be very close to $1$). 

(ii)  Take  $\varepsilon= 1-2\kappa-d/p$, then we get
$$
\sup_{s<\tau} |u(s,x)| \leq N  \rho^{\varepsilon}(x) \to 0
$$
substantially fast as $\rho(x)\to 0$. Hence $u$ vanishes on the boundary, and this is a reason we do not need to explicitly impose the zero boundary condition to the equation.

(iii) Since $\gamma-\varepsilon'>0$, from (\ref{holder main}) it only follows that $u_x$ is H\"older continuous in compact subsets of $\cO$.
\end{remark}

\mysection{Some  auxiliary results related to linear equations}

In this section, we collect  a few results related  to the following linear equation:
\begin{align}
du 
					\notag
&= \Big[D_i\Big(a^{ij}(t,x)u_{x^j}+b^i(t,x) u + f^i\Big)+ \bar{b}^i(t,x) u_{x^i} +c(t,x)u +f\Big]dt  \\
					&\quad +  \Big[\nu^k(t,x)u+g^k\Big] dW^k_t, \qquad t \leq \tau; \quad u(0,\cdot)=u_0.
					\label{linear eqn}
\end{align}

\begin{assumption}
                                                          \label{as li el}
  (i)  The coefficients $a^{ij}(t,x)$, $b^i(t,x)$, $\bar{b}^i(t,x)$, $c(t,x)$, and $\nu^k(t,x)$ 
are  $\cP \times \cB(\cO) $-measurable  functions.
  
  (ii) There exists a constant $\delta_0>0$ 
 such that
\begin{equation}
 \delta_0 |\xi|^2 \leq   a^{ij}\xi^i \xi^j
 \leq \delta^{-1}_0 |\xi|^2,  
\end{equation}
for all $\omega,t,x$ and $\xi\in \fR^d$. 
\end{assumption}

\begin{assumption}
               \label{as li bd}
  For all $\omega,t,x$,
  $$
\rho(x)\left[|b^i(t,x)|+|\bar{b}^i(t,x)|+|\nu(t,x)|_{l_2}\right]+\rho^2(x)|c(t,x)|\leq \delta^{-1}_0,
$$
and there is a control on the blow up of  the coefficients near the boundary:
  $$
  \rho(x)|b^i(t,x)|+\rho(x)|\bar{b}^i(t,x)|+\rho^2(x)|c(t,x)|+\rho(x)|\nu(t,x)|_{\l_2} \to 0
  $$
  as $\rho(x)\to 0$. In other words, there exists a nondecreasing function
  $\pi_0 : [0,\infty) \mapsto [0,\infty)$ such that $\pi_0(t) \downarrow 0$ as $t \downarrow 0$ and
 $$
 \rho(x)|b^i(t,x)|+\rho(x)|\bar{b}^i(t,x)|+\rho^2(x)|c(t,x)|+\rho(x)|\nu(t,x)|_{\l_2}\leq \pi_0(\rho(x)).
 $$
 \end{assumption}

\begin{remark}
Obviously Assumption \ref{as li bd}  holds if the coefficients are bounded.  It also holds if 
\begin{align}
					\label{0213 eqn 1}
|b^i(t,x)|+|\bar{b}^i(t,x)|+|\nu(t,x)|_{l_2}\leq N\rho^{-1+\varepsilon}(x), \quad |c(t,x)|\leq N\rho^{-2+\varepsilon}(x),
\end{align}
for some $\varepsilon, N>0$. Note that (\ref{0213 eqn 1}) allows the coefficients to blow up substantially fast near the boundary.

\end{remark}


\begin{assumption}
					\label{uni as}
$a^{ij}$ are uniformly continuous in $x$, 
that is, there exists a nondecreasing function $\pi_0: [0,\infty) \to [0,\infty)$ such that
\begin{align*}
|a^{ij}(t,x) - a^{ij}(t,y)| \leq \pi_0(|x-y|), \quad \forall \omega,t
\end{align*}
and $\pi_0(\lambda)  \to 0$ as $\lambda \to 0$.
\end{assumption}

Fix $\kappa_0\in (0,1)$, and for $\gamma\geq 0$, denote $\gamma_+=\gamma$ if $\gamma$ is integer and otherwise $\gamma_+=\gamma+\kappa_0$.

\begin{theorem}
                                              \label{lin thm}
Let $p\geq 2$,  $\tau \leq T$ be a stopping time, and 
\begin{align}
						\label{theta li ra}
d-1+p
< \theta < d-1+p.
\end{align}
Suppose that Assumptions  \ref{as li el}, \ref{as li bd}, and \ref{uni as} hold, and there exists a constant $\bar K$ such that
\begin{align}
					\label{gam re}
&|a^{ij}(t,\cdot)|^{(0)}_{\gamma_+}+|\psi \bar{b}^i(t,\cdot)|^{(0)}_{\gamma_+}
+|\psi \nu(t,\cdot)|^{(0)}_{\gamma_+} 
+|\psi b^i(t,\cdot)|^{(0)}_{\gamma_+}+|\psi^2 c(t,\cdot)|^{(0)}_{\gamma_+} 
\leq \bar K \quad \forall \omega, t.
\end{align}
Then  for any $f^{i}\in \bH^{\gamma}_{p,\theta}(\cO,\tau)$,
$f\in \psi^{-1} \bH^{\gamma-1}_{p,\theta}(\cO,\tau)$,
$g \in\bH^{\gamma}_{p,\theta}(\cO,\tau,l_2)$, and
$u_0 \in U^{\gamma+1}_{p,\theta}(\cO)$
 equation (\ref{linear eqn}) with  initial data $u_0$ has
a unique solution $u $ in $\frH^{\gamma+1}_{p,\theta}(\cO,\tau)$ and
\begin{equation}
                                               \label{8.17.130}
\|u\|_{\frH^{\gamma+1}_{p,\theta}(\cO,\tau)} \leq
N ( \|f^{i}\|_{\bH^{\gamma}_{p,\theta}(\cO,\tau)} +
\|\psi f\|_{\bH^{\gamma-1}_{p,\theta}(\cO,\tau)}
+ \|g\|_{\bH^{\gamma}_{p,\theta}(\cO,\tau,l_2)} +
 \|u_{0}\|_{U^{\gamma+1}_{p,\theta}(\cO)}),
\end{equation}
where $N$ depends only on $d$, $p$, $\gamma$, $\delta_0$,  $T, \cO$, and the function $\pi_0(t)$.
\end{theorem}

\begin{proof}
See \cite[Theorem 3.13]{kim2014some}. We only mention that the result of this theorem was first proved by Krylov and Lototsky \cite{KL2} when $\cO=\fR^d_+$ and the coefficients are independent of $x$. Then the result   was extended to general $C^1$-domains  in \cite{Kim04} based on   localization and  flattening the boundary arguments.  In \cite{Kim04} the coefficients $a^{ij}$ are  continuous in $x$ and consequently we only have $u\in \frH^1_{p,d}(\cO,\tau)$.  Finally better regularity of the solution is obtained in  \cite{kim2014some} under  H\"older continuity (\ref{gam re}).
 \end{proof}


\begin{corollary}[$\theta=d$ with no stochastic term]
					\label{lin cor}
Assume $\nu^k=g^k=0$ for each $k$ and  let $u$ be the solution in  Theorem \ref{lin thm} corresponding to the case $\theta=d$, that is, $u$ is the solution to 
\begin{align*}
du 
&= \Big[D_i\Big(a^{ij}(t,x)u_{x^j}+b^i(t,x) u + f^i\Big)+ \bar{b}^i(t,x) u_{x^i} +c(t,x)u +f\Big]dt,  \quad t\leq \tau, \\
&  \quad 
u(0,\cdot)=u_0.
\end{align*}
Then (a.s.)
\begin{align*}
&\|\psi^{-1} u\|_{L_p\left([0,\tau];H_{p,d}^{1 + \gamma}(\cO,\tau)\right)} +\|\psi \bD u\|_{L_p([0,\tau];H^{\gamma-1}_{p,d}(\cO)} \\
&\leq N\left( \|f^{i}\|_{L_p\left([0,\tau];H_{p,d}^{\gamma}(\cO,\tau)\right)} +
\|\psi f\|_{L_p\left([0,\tau];H_{p,d}^{\gamma -1}(\cO,\tau)\right)}
+\|u_{0}\|_{H_{p,d}^{1+\gamma-2/p}} \right),
\end{align*}
where $N=N(d,p,\gamma,\delta_0,T, \pi_0,\cO)$ is independent of $\omega$.
\end{corollary}
\begin{proof}
It is enough to fix $\omega$, and then apply Theorem \ref{lin thm} to the corresponding deterministic equation.
\end{proof}

\begin{theorem}[$p=2$ and $\theta=d$ with only measurable  coefficients]
                                                   \label{l2 thm}
Let $\tau\leq T$ and  Assumptions  \ref{as li el} and \ref{as li bd} hold.
Then  for any $f^{i}\in \bL_{2,d}(\cO,\tau)$,
$f\in \psi^{-1} \bH^{-1}_{2,d}(\cO,\tau)$,
$g \in\bL_{2,d}(\cO,\tau,\l_2)$, and
$u_0 \in U^{1}_{2,d}(\cO)$
 equation (\ref{linear eqn}) with initial data $u_0$
has a unique solution $u \in \frH^1_{2,d}(\cO,\tau)$, and for this solution
\begin{align*}
\|u\|_{\frH^{1}_{2,d}(\cO,\tau)} 
\leq
N ( \|f^{i}\|_{\bL_{2,d}(\cO,\tau)} +
\|\psi f\|_{\bH^{-1}_{2,d}(\cO,\tau)}
+ \|g\|_{\bL_{2,d}(\cO,\tau,\l_2)} +
 \|u_{0}\|_{U^{1}_{2,d}(\cO)}),
\end{align*}
where $N$ depends only on $\delta_0$, $T, \cO$,  and the function $\pi_0$.
\end{theorem}

\begin{proof}
This is a very classical result (see e.g. \cite{Ro}) if the coefficients are bounded.  See  \cite[Theorem 2.19]{kim2009sobolev} for the general case. We remark that in our main theorem, Theorem \ref{main thm}, the coefficients are assumed to be bounded. Hence,  the classical result of \cite{Ro} is enough for our need. 

\end{proof}

\begin{theorem}
                   \label{hol thm}
  Let  $u\in \frH^1_{2,d}(\cO,\tau)$ be the solution taken from Theorem  \ref{l2 thm} and assume  $\nu^k=0$ for each $k$.
Assume
\begin{equation}
                              \label{eqn bounded bc}
|b^i(t,x)| + |\overline{b}(t,x)| + |c(t,x)| \leq \bar K \quad \forall \omega,t,x,
\end{equation}
\begin{equation}
                           \label{eqn on fg}
 f^i\in \bL_{p,d}(\cO,\tau), \quad f\in \psi^{-1}\bH^{-1}_{p,d}(\cO,\tau), \quad g\in \bL_{p,d}(\cO,\tau,l_2), \nonumber
 \end{equation}
for some $p>d+2$.
  Then there exists a constant $\bar{\alpha}>0$ 
so that  if  $\alpha < \bar \alpha$ and $u_0\in L_p(\Omega,C^{\alpha}(\cO))$ then
\begin{align}
                                                     \label{5151}
\bE |u|^p_{C^\alpha ([0,\tau] \times \cO)} 
&\leq N
\|f^{i}\|^p_{\bL_{p,d}(\cO,\tau)} + N\|\psi
f\|_{\bH^{-1}_{p,d}(\cO,\tau)}   \\
&+N
\|g\|^p_{\bL_{p,d}(\cO,\tau,l_2)}+ N \bE|u_0|^p_{C^{\alpha}(\cO)},\nonumber
\end{align}
where $N$ depends only on $d$, $p, \alpha$,  $\delta_0$,  $\bar K$, $T$, and $\cO$.
\end{theorem}
\begin{proof}
See \cite[Theorem 2.9]{kim2014some} or \cite[Theorem 2.4]{kim2004lq} for detailed proof.  Below we only give a skecth of the proof. 
For simplicity assume $b^i=\bar{b}^i=c=u_0=0$.  By Theorem \ref{lin thm}, there is a unique solution $v\in \frH^1_{p,d}(\cO,\tau)$  to
$$
dv=(\Delta v +f)dt +g^k dW^k_t, \quad v(0)=0.
$$
By  (\ref{8.17.130}) and Corollary \ref{lin cor} with $\gamma=0$, there is  $\alpha_1>0$ so that 
$$
\bE \|\nabla v\|^p_{L_p([0,\tau]\times \cO)}+ \bE |v|^p_{C^{\alpha_1}([0,\tau]\times \cO)}\leq N ( \|\psi f\|^p_{\bH^{-1}_{p,d}(\cO,\tau)}+\|g\|^p_{\bL_{p,d}(\cO,\tau,l_2)}).
$$
Note that, for each fixed $\omega$, the function $\bar{u}:=u-v$ satisfies the deterministic equation
$$
\frac{d\bar{u}}{dt}=D_i(a^{ij}\bar{u}_{x^i}+\bar{f^i}), \quad \bar{u}(0,\cdot)=0,
$$
where $\bar{f}^i=(a^{ij}-\delta^{ij})v_{x^j}+f^i$. Here $\delta^{ij}$ is the Kronecker delta, i.e. $\delta^{ij}=1$ if $i = j$ and otherwise $\delta^{ij}=0$. Then using a classical result for the deterministic equation (e.g. \cite{LSU}), for some $\alpha_2>0$ we have
$$
|\bar{u}|_{C^{\alpha_2}([0,\tau]\times \cO)} \leq N \|\bar{f}^i\|_{L_p([0,\tau]\times \cO)},
$$
where $N$ is independent of $\omega\in \Omega$.  Combining above two estimates we get (\ref{5151}) with $\bar{\alpha}=\alpha_1 \wedge \alpha_2$.
\end{proof}

\mysection{local solution}

In this section, we construct a nonzero stopping time $\tau' \leq \tau$ so that the equation
\begin{align}
du = \Big(D_i \Big[a^{ij}(u)u_{x^j}+b^i(u) u + f^i\Big] + \bar{b}^i(u) u_{x^i} +c(u)u +f\Big)dt+g^k dW_t^k
					\label{nu zero}
\end{align}
has a solution for $t\leq \tau'$.
\begin{lemma}
					\label{local sol}
Suppose that Assumptions \ref{as mea}-\ref{as co con} hold
and assume that $\tau$ is a nonzero stopping time.
Let $f^i\in \bL_{p,d}(\cO,\tau)$, $f\in \bH^{-1}_{p,d}(\cO,\tau)$, $g\in \bL_{p,d}(\cO,\tau,l_2)$ and $u_0\in U^1_{p,d}(\cO)$.  Then there exist a nonzero stopping time $\tau'\leq \tau$ 
such that equation (\ref{nu zero}) with initial data $u_0$ has a  solution $u$ in $\frH^{1}_{p,d}(\cO,\tau')$.
\end{lemma}

\begin{proof}
{\bf Step 1.} We prove the lemma if  $K_3$ is  sufficiently close to zero.

%
%

\vspace{2mm}

Let $u^1  \in \frH_{p,d}^{1+\gamma_0}(\cO, \tau)$ be the solution of the following equation:
\begin{align}
					\label{eq u1}
du&= \Big[D_i\Big(\delta^{ij}u_{x^j}+  f^i\Big) +f\Big]dt +  g^k dW_t^k,  \qquad t\leq \tau, 
 \quad   u(0, \cdot)=u_0.
\end{align}
By  the assumption $\gamma_0 \in \left(\frac{d+2}{p},1\right)$, one can choose $\gamma_1$ and $\gamma_2$ so that
\begin{align}
					\label{0211 eqn 1}
\frac{d+2}{p}<\gamma_2 < \gamma_1< \gamma_0< 1.
\end{align}
Since $p>d+2$,
$$
\hat{\alpha}=\hat{\alpha}(p):=1-d/p-2/p>0.
$$
Let  $\alpha \in (0,\hat \alpha)$.  Then
by \eqref{eqn 10.20} and Corollary \ref{hol cor},
$$
\bE |u^1|^p_{C([0,\tau]; C^{\alpha}(\cO))} + \bE | \psi Du^1 |^p_{C([0,\tau]\times \cO)}  
\leq N\|u^1\|^p_{\frH_{p,d}^{1+\gamma_0}(\cO, \tau)} < \infty.
$$
Denote
\begin{align*}
A_n:=&\left\{\omega\in \Omega: |u^1|_{C([0,\tau]; C^{\alpha}(\cO))} + | \psi Du^1 |_{C([0,\tau]\times \cO)} < n \right\}.
\end{align*}
Then $P(\cup_{n=1}^\infty A_n)=1$, and thus
we can fix $n_0 \in \bN$ such that 
\begin{align*}
&P\left( \{ \tau \neq 0 \} \cap A_{n_0} \right) > 0.
\end{align*}
Define
\begin{align*}
\tau^{''}:=\inf\left\{ t \leq \tau : |u^1|_{C([0,t]; C^{\alpha}(\cO))}   + | \psi Du^1 |_{C([0,t]\times \cO)}
\geq  n_0 \right\},
\end{align*}
\begin{align*}
\tau^{'''}:= \inf\{ t \leq \tau : \|\psi^{-1}u^1\|_{L_p([0,t] ; H^{1+\gamma_0}_{p,d}(\cO))} > \varepsilon \},
\end{align*}
and
\begin{align}
\tau' := \tau^{''} \wedge \tau^{'''},
\end{align}
where $\varepsilon \in (0,1)$  will be specified later.  
It is obvious that $\tau'''>0$ (a.s.) and  $\tau'$ is a nonzero stopping time. The latter is because $\tau''=\tau$ on $A_{n_0}$.

Denote
\begin{align*}
\Phi(\tau')
&:= \Big\{ u \in \frH_{p,d}^{1 + \gamma_2}(\cO, \tau') :
\|\psi^{-1}(u- u^1)\|_{L_p\left([0,\tau'];H_{p,d}^{1 + \gamma_2}(\cO,\tau')\right)} \leq 1 \,\,(a.s.), \\
& \quad \quad \,
 |u-u^1 |_{C([0,\tau']\times \cO)}  + | \psi D(u-u^1) |_{C([0,\tau']\times \cO)} 
 \leq 1 \,\,(a.s.),~u(0,\cdot)=u_0
 \Big\}.
\end{align*}
For each $v \in \Phi(\tau')$, by $\cR v$ we denote the solution in $\frH^1_{p, d}(\cO,\tau')$ 
to  the equation 
\begin{align*}
du 
&= \Big[ D_i\Big(a^{ij}(u^1)u_{x^j} + b^i(u^1)u+[a^{ij}(v)-a^{ij}(u^1)]v_{x^j}+ [b^i(v)-b^i(u^1)]v+f^i\Big) \\
&\quad \quad + \bar b^i(u^1)u_{x^i}+ c(u^1)u + [\bar b^i (v)- \bar b^i (u^1)]v_{x^i}+[c(v) -c(u^1)]v + f\Big]dt  \\
&\quad +  g^k W_t^k, \qquad  t \leq \tau';   \quad u(0,\cdot)=u_0.
\end{align*}
 The map $v\to \cR v \in \frH^1_{p,d}(\cO,\tau')$ is well-defined due to Theorem \ref{lin thm} since $a^{ij}(u^1)$ is uniformly continuous in $x$ (uniformly in $(\omega,t)$) and other coefficients are assumed to be bounded.  Indeed,
\begin{eqnarray*}
&& |a^{ij}(t,x,u^1(t,x))-a^{ij}(t,y,u^1(t,y))|\\
&&\leq K_2|x-y|^{\gamma_0} + K_3 |u^1(t,x)-u^1(t,y)|\leq K_2|x-y|^{\gamma_0}+K_3n_0|x-y|^{\alpha}, \quad t\leq \tau'.
\end{eqnarray*}

To check  $\cR v \in \frH^{1+\gamma_2}_{p,d}(\cO,\tau')$, first note that by  \eqref{0222 eqn 1}, \eqref{0213 eqn 1}, and Assumption \ref{as bd 2},
\begin{align*}
&\sup_{t \leq \tau}|a(t,\cdot,u^1(\cdot))|^{(0)}_{\gamma_1} 
+\sup_{t \leq \tau}|\psi b^i(t,\cdot,u^1(\cdot))|^{(0)}_{\gamma_1} \\
&\quad + \sup_{t \leq \tau}|\psi \bar b^i(t,\cdot,u^1(\cdot))|^{(0)}_{\gamma_1}
+\sup_{t \leq \tau}|\psi^2 c(t,\cdot,u^1(\cdot))|^{(0)}_{\gamma_1} 
\leq N n_0,
\end{align*}
where $N$ is a constant depending only on $d$,  $K_1$, and $K_2$. Thus by Theorem \ref{lin thm}  we conclude $\cR  v\in \frH^{1+\gamma_2}_{p,d}(\cO,\tau')$.  Here we used the existence result in $\frH^{1+\gamma_2}_{p,d}(\cO,\tau')$ and the uniqueness result in $\frH^1_{p,d}(\cO,\tau')$.

Next we show $\cR v\in \Phi(\tau')$ if $K_3$ is sufficiently small.  
Note that $(\cR v-u^1)(0,\cdot)=0$ and
\begin{align*}
d(\cR v - u^1) 
&=\Big[D_i\Big(a^{ij}(u^1)(\cR v - u^1)_{x^j} + b^i(u^1)(\cR v - u^1)+\tilde f^i\Big) \\
&\quad \,\, + \bar b^i (u^1) ( \cR v - u^1)_{x^i} + c(u^1)(\cR v - u^1) + \tilde f\Big]dt, \quad  \quad t\leq \tau',
\end{align*}
where
\begin{align*}
\tilde f^i 
&:= [a^{ij}(v)-a^{ij}(u^1)]v_{x^j} + [b^i(v)-b^i(u^1)]v+ a^{ij}(u^1)u^1_{x^j}+b (u^1) u^1 -\delta^{ij}u^1_{x^j}\\
&= [a^{ij}(v)-a^{ij}(u^1)](v_{x^j}-u^1_{x^j})+[a^{ij}(v)-a^{ij}(u^1)]u^1_{x^j} + a^{ij}(u^1)u^1_{x^j}\\
&\quad  + [b^i(v)-b^i(u^1)](v-u^1)+  [b (v)-b^i(u^1)] u^1 
  + b^i(u^1) u^1  - \delta^{ij}u^1_{x^j}
\end{align*}
and
\begin{align*}
\tilde f 
&:= [\bar b^i (v)- \bar b^i (u^1)]v_{x^i} +[c(v) -c(u^1)]v +\bar b^i(u^1){u^1_{x^i}}+ c(u^1)u^1 \\ 
&= [\bar b^i (v)- \bar b^i (u^1)](v-u^1)_{x^i} +[c(v) -c(u^1)](v-u^1) \\
&\quad +[\bar b^i (v)- \bar b^i (u^1)]u^1_{x^i}+[c(v) -c(u^1)]u^1+c(u^1)u^1+\bar b^i(u^1)u^1_{x^i}+c(u^1)u^1.
\end{align*}

\noindent
 Using the deterministic version of   \eqref{embed eqn} with $\gamma=\gamma_2$ and Corollary \ref{lin cor},  we get (a.s.)
\begin{align}
 \nonumber
&\|(\cR v-u^1)\|_{C\left([0,\tau']\times \cO \right) }
+\|\psi (\cR v-u^1)_x\|_{C\left([0,\tau']\times \cO \right) } \\
\nonumber
&+\|\psi^{-1}(\cR v- u^1)\|_{L_p\left([0,\tau'];H_{p,d}^{1 + \gamma_2}(\cO,\tau')\right)}\\
&\leq N \left( \|\tilde f^{i}\|_{L_p\left([0,\tau'];H_{p,d}^{\gamma_2}(\cO,\tau')\right)} 
+\|\psi \tilde f\|_{L_p\left([0,\tau'];H_{p,d}^{-1+ \gamma_2}(\cO,\tau')\right)} \right), \label{eqn 3.15.1}
\end{align}
where $N=N(n_0,d,p,\delta_0,T,K_1,K_2)$.   

To estimate $\tilde{f}^i$ and $\psi \tilde{f}$ in (\ref{eqn 3.15.1}), we show that 
 for any $\delta>0$,
\begin{align}
 \label{eqn 3.15.2}
|a^{ij}(v)-a^{ij}(u^1)|^{(0)}_{\gamma_1}
\leq  N(d,\cO)\left(K_3(1+  \delta^{-\gamma_1} + n_0) +K_2 \delta^{\gamma_0 -\gamma_1} \right).
\end{align}
First observe
$$
 |a^{ij}(v)-a^{ij}(u^1)|\leq K_3|v-u^1|\leq K_3, \quad \quad t\leq \tau'.
 $$
 If $|x-y|\geq \delta$, 
\begin{align*}
&\rho^{\gamma_1}(x,y)\frac{|a^{ij}(x,v(x)) -a^{ij}(x,u^1(x))-\left(a^{ij}(y,v(y)) -a^{ij}(y,u^1(y))\right)|}{|x-y|^{\gamma_1}} \\
&\leq \rho^{\gamma_1}(x,y)\frac{|a^{ij}(x,v(x)) -a^{ij}(x,u^1(x))|}{|x-y|^{\gamma_1}}
+\rho^{\gamma_1}(x,y)\frac{|a^{ij}(y,v(y)) -a^{ij}(y,u^1(y))|}{|x-y|^{\gamma_1}} \\
&\leq 2NK_3|v-u^1|_{C(\cO)}\frac{1}{|x-y|^{\gamma_1}} 
\leq 2NK_3\delta^{-\gamma_1},
\end{align*}
and if $|x-y|\leq \delta$
\begin{align*}
&\rho^{\gamma_1}(x,y)\frac{|a^{ij}(x,v(x)) -a^{ij}(x,u^1(x))-\left(a^{ij}(y,v(y)) -a^{ij}(y,u^1(y))\right)|}{|x-y|^{\gamma_1}} \\
&\leq \rho^{\gamma_1}(x,y)\left|\frac{a^{ij}(x,v(x)) -a^{ij}(x,v(y))}{|x-y|^{\gamma_1}}
-\frac{a^{ij}(x,u^1(x)) -a^{ij}(x,u^1(y))}{|x-y|^{\gamma_1}} \right|\\
&\quad+ \rho^{\gamma_1}(x,y)\left|\frac{a^{ij}(x,v(y)) -a^{ij}(y,v(y))}{|x-y|^{\gamma_1}}
-\frac{a^{ij}(x,u^1(y)) -a^{ij}(y,u^1(y))}{|x-y|^{\gamma_1}} \right|\\
&\leq K_3( [v]^{(0)}_{\gamma_1}+[u^1]^{(0)}_{\gamma_1}) +NK_2 |x-y|^{\gamma_0 -\gamma_1} \\
&\leq K_3( [v]^{(0)}_{\gamma_1}+[u^1]^{(0)}_{\gamma_1}) +NK_2 \delta^{\gamma_0 -\gamma_1}.
\end{align*}
Hence (\ref{eqn 3.15.2}) is proved. Similarly, for any $\delta>0$,
\begin{align*}
\left|\psi [b^i(v)-  b^i(u^1)] \right|^{(0)}_{\gamma_1}
&\leq  N\left(K_3(1+ \delta^{-\gamma_1}+ n_0) +K_2 \delta^{\gamma_0 -\gamma_1} \right)=: N I(\delta,K_3).
\end{align*}
Therefore by Lemma \ref{wei lem}(iv), 
\begin{eqnarray*}
&&\|(a^{ij}(v)-a^{ij}(u^1))\cdot (v_{x^j}-{u^1_{x^j}})\|_{L_p\left([0,\tau'];H_{p,d}^{\gamma_2}(\cO,\tau')\right)}\\
&&\leq N I(\delta,K_3)\|(v-u^1)_x\|_{L_p\left([0,\tau'];H_{p,d}^{\gamma_2}(\cO,\tau')\right)}\\
&&\leq 
N I(\delta,K_3)\|\psi^{-1}(v-u^1)\|_{L_p\left([0,\tau'];H_{p,d}^{1+\gamma_2}(\cO,\tau')\right)} \leq NI(\delta,K_3).
\end{eqnarray*}
Similarly,
\begin{eqnarray*}
&&\|[a^{ij}(v)-a^{ij}(u^1)]{u^1_{x^j}}\|_{L_p\left([0,\tau'];H_{p,d}^{\gamma_2}(\cO,\tau')\right)}
\leq NI(\delta,K_3)
\|u^1_x\|_{L_p\left([0,\tau'];H_{p,d}^{\gamma_2}(\cO,\tau')\right)}\\
&& \leq NI(\delta,K_3)\|\psi^{-1}u^1\|_{L_p\left([0,\tau'];H_{p,d}^{1+\gamma_0}(\cO,\tau')\right)}\leq N I(\delta,K_3) \varepsilon.
\end{eqnarray*}
In this way, we get
\begin{align*}
 &\|\tilde f^{i}\|_{L_p\left([0,\tau'];H_{p,d}^{\gamma_2}(\cO,\tau')\right)} \\
 &\leq\left\|[a^{ij}(v)-a^{ij}(u^1)](v_{x^j}-{u^1_{x^j}})\right\|_{L_p\left([0,\tau'];H_{p,d}^{\gamma_2}(\cO,\tau')\right)}\\
&\quad +\left\|[a^{ij}(v)-a^{ij}(u^1)]{u^1_{x^j}}\right\|_{L_p\left([0,\tau'];H_{p,d}^{\gamma_2}(\cO,\tau')\right)}\\
&\quad + \left\|\psi[b^i(v)-b^i(u^1)]\psi^{-1}(v-u^1)\right\|_{L_p\left([0,\tau'];H_{p,d}^{\gamma_2}(\cO,\tau')\right)} \\
&\quad +\left\|\psi [b (v)-b^i(u^1)] \psi^{-1} u^1\right\|_{L_p\left([0,\tau'];H_{p,d}^{\gamma_2}(\cO,\tau')\right)} \\
&\quad +\left\|a^{ij}(u^1){u^1_{x^j}}\right\|_{L_p\left([0,\tau'];H_{p,d}^{\gamma_2}(\cO,\tau')\right)}
+ \left\|\psi b^i(u^1) \psi^{-1}u^1\right\|_{L_p\left([0,\tau'];H_{p,d}^{\gamma_2}(\cO,\tau')\right)} \\
&\quad+ \left\|\delta^{ij}{u^1_{x^j}}\right\|_{L_p\left([0,\tau'];H_{p,d}^{\gamma_2}(\cO,\tau')\right)} 
\leq N(I(\delta,K_3)+\varepsilon).
\end{align*}
Similarly,
\begin{align*}
\|\psi \tilde f\|_{L_p\left([0,\tau'];H_{p,d}^{-1+ \gamma_2}(\cO,\tau')\right)}
\leq N(I(\delta,K_3)+\varepsilon).
\end{align*}
Therefore taking sufficiently small $\varepsilon$ and  $\delta$, and then assuming $K_3$ is very small, we get
\begin{align}
					\label{0127 eq 1}
\cR v \in \Phi(\tau') \qquad \forall v \in \Phi(\tau').
\end{align}

Next we claim that the operator $\cR$ becomes a contraction mapping on 
$$
\Phi(\tau') =\Phi(\tau') \cap {\frH^{1}_{p,d}(\cO,\tau')}
$$
with respect to the norm $\|\cdot\|_{\frH^1_{p,d}(\cO,\tau')}$ 
if $K_3$ and $\varepsilon$ are small enough.
We may assume that $K_3$ and $\varepsilon$ are small so that \eqref{0127 eq 1} holds.
Observe that for each $v,w \in  \Phi(\tau')$,  $(\cR v - \cR w)(0, \cdot)=0$ and
\begin{align*}
d(\cR v - \cR w)  
&=\Big[D_i\Big(a^{ij}(u^1)(\cR v - \cR w)_{x^j} + b^i(u^1)(\cR v - \cR w)+\bar f^i\Big) \\
&\quad + \bar b^i (u^1) ( \cR v - \cR w)_{x^i} + c(u^1)(\cR v - \cR w) + \bar f\Big]dt, \qquad t\leq \tau',
\end{align*}
where
\begin{align*}
\bar f^i 
&:= [a^{ij}(v)-a^{ij}(u^1)]v_{x^j} + [b^i(v)-b^i(u^1)]v\\
&\quad -[a^{ij}(w)-a^{ij}(u^1)]w_{x^j} - [b(w)-b^i(u^1)]w \\
&= [a^{ij}(v)-a^{ij}(u^1)](v-w)_{x^j} + [b^i(v)-b^i(u^1)](v-w) \\
&\quad +[a^{ij}(v)-a^{ij}(w)](w-u^1)_{x^j} + [b^i(v)-b(w)](w-u^1) \\
&\quad +[a^{ij}(v)-a^{ij}(w)]{u^1_{x^j}} + [b^i(v)-b(w)]u^1
\end{align*}
and
\begin{align*}
\bar f 
&:= [\bar b^i (v)- \bar b^i (u^1)]v_{x^i}+[c(v) -c(u^1)]v   \\
&\quad -[\bar b^i (w)- \bar b^i (u^1)]w_{x^i}-[c(w) -c(u^1)]w \\
&= [\bar b^i (v)- \bar b^i (u^1)](v-w)_{x^i}+[c(v) -c(u^1)](v-w)   \\
&\quad +[\bar b^i (v)- \bar b^i (w)](w-u^1)_{x^i}+[c(v) -c(w)](w-u^1) \\
&\quad +[\bar b^i (v)- \bar b^i (w)]{u^1_{x^i}}+[c(v) -c(w)]u^1.
\end{align*}
By Theorem \ref{lin thm},
\begin{align*}
&\|\cR v - \cR w\|_{\frH^{1}_{p,d}(\cO,\tau')}  
\leq N \left( \|\bar f^{i}\|_{\bL_{p,d}(\cO,\tau')} + \|\psi \bar f\|_{\bH^{-1}_{p,d}(\cO,\tau')} \right).
\end{align*}
Since 
$$|a^{ij}(v)-a^{ij}(u^1)|+|\psi [b^i(v)-b^i(u^1)]|\leq K_3|v-u^1|\leq K_3,
$$
it follows that
\begin{eqnarray*}
&&\|(a^{ij}(v)-a^{ij}(u^1))(v-w)_{x^j}\|_{\bL_{p,d}(\cO,\tau')} 
+ \|\psi (b^i(v)-b^i(u^1))\psi^{-1}(v-w)\|_{\bL_{p,d}(\cO,\tau')}\\
&&\leq N K_3 (\|(v-w)_x\|_{\bL_{p,d}(\cO,\tau')} +\|\psi^{-1} (v-w)\|_{\bL_{p,d}(\cO,\tau')} )\\
&& \leq NK_3 \|v-w\|_{\frH^1_{p,d}(\cO,\tau')}.
\end{eqnarray*}
Also, using 
$$
\sup_x \left(|w(t,x)-u^1(t,x)|+|\psi(w-u^1)_x(t,x)|+|u^1(t,x)|+|\psi u^1_x(t,x)|\right)\leq 2+n_0
$$
for $t\leq \tau'$,  we get
\begin{eqnarray*}
&& \|[a^{ij}(v)-a^{ij}(w)] (w-u^1)_{x^j}\|_{\bL_{p,d}(\cO,\tau')}= \|\psi^{-1}[a^{ij}(v)-a^{ij}(w)] \psi(w-u^1)_{x^j}\|_{\bL_{p,d}(\cO,\tau')}\\
&&\leq 2\|\psi^{-1}(a^{ij}(v)-a^{ij}(w))\|_{\bL_{p,d}(\cO,\tau)}\leq 2K_3 \|\psi^{-1}(v-w)\|_{\bL_{p,d}(\cO,\tau')},
\end{eqnarray*}
and similarly
\begin{eqnarray*}
&&\| [b^i(v)-b(w)] (w-u^1)\|_{\bL_{p,d}(\cO,\tau')} 
+\|[a^{ij}(v)-a^{ij}(w)]  u^1_{x^j}\|_{\bL_{p,d}(\cO,\tau')} \\
&&+ \| [b^i(v)-b(w)]u^1\|_{\bL_{p,d}(\cO,\tau')} 
\leq NK_3\|v-w\|_{\frH^1_{p,d}(\cO,\tau')}.
\end{eqnarray*}
Hence,
\begin{align*}
\|\bar f^{i}\|_{\bL_{p,d}(\cO,\tau')} \leq NK_3\|v-w\|_{\frH^1_{p,d}(\cO,\tau')},
\end{align*}
where $N$ depends only on $d$, $p$, $\gamma$, $\delta_0$,  $\bar K$, $T$, $\cO$, and $n_0$.
Furthermore, based on the same computations, we also get
\begin{align*}
\|\psi \bar f\|_{\bH^{-1}_{p,d}(\cO,\tau')}
\leq N\| \psi \bar f\|_{\bL_{p,d}(\cO,\tau')}
\leq NK_3\|v-w\|_{\frH^1_{p,d}(\cO,\tau')}.
\end{align*}
Therefore, taking $K_3$ small enough, we obtain
\begin{align}
						\label{eqn 0127 2}		
\|\cR v - \cR w\|_{\frH^{1}_{p,d}(\cO,\tau')}  
\leq \frac{1}{2}\|v - w\|_{\frH^{1}_{p,d}(\cO,\tau')}
\qquad \forall v,w \in  \Phi(\tau').
\end{align}

For $n=2,3,\cdots$, define $u^{n+1} = \cR u^n$ inductively. Then $\{u^n:n=1,2,\cdots\}$ becomes  a Cauchy sequence in $\frH^1_{p,d}(\cO,\tau')$. Let $u$ be the limit of $u^n$ in $\frH^{1}_{p,d}(\cO,\tau')$.
Then using the relation
\begin{align*}
 du^{n+1}  
&= \Big[D_i\Big(a^{ij}(u^1)u^{n+1}_{x^j} +[a^{ij}(u^n)-a^{ij}(u^1)]u^n_{x^j}+ b^i(u^1)u + [b^i(u^n)-b^i(u^1)]u^n \Big)\\
 &\quad +D_i f^i+
 \bar b^i(u^1)u^{n+1}_{x^i}+ [\bar b^i (u^n)- \bar b^i (u^1)]u^n_{x^i}+ c(u^1)u^{n+1}    \\
&\quad +[c(u^n) -c(u^1)]u^n+ f\Big]dt +  g^k dW_t^k, \qquad t \leq \tau',
\end{align*}
and taking $n\to \infty$, we find $u(0,\cdot)=u_0$ and the equality
\begin{align*}
du= \Big[D_i\Big(a^{ij}(u) {u}_{x^j}+ b(u) u+f^i\Big) + \bar b^i (u){u}_{x^i}+ c(u)u + f\Big]dt  +  g^k dW_t^k
\end{align*}
holds for almost all $t\leq \tau'$ (a.s.).  It follows that the above equality holds for all $t\leq \tau'$ (a.s.)
since both sides above are continuous in $t$ due to Theorem \ref{embedding}.
\vspace{2mm}

{\bf Step 2.} We remove the condition that $K_3$ is very small. 
\vspace{2mm}

For $\delta>0$, 
consider the transform $u(t,x) \mapsto v(t,x):=\frac{1}{\delta}u(t,x)$.
Then $u$ is a solution of \eqref{nu zero} if and only if
\begin{align}
dv 
					\notag
&= \Big[D_i \Big[\tilde a^{ij}(v)u_{x^j}+\tilde b^i(v) v + \tilde f^i\Big] + \tilde{\bar{b}}^i(v) v_{x^i} +\tilde c(v)v +f\Big]dt+g^k dW^k, \\
					\label{0210 eqn v}
&~ t \leq \tau', \quad v(0)=u_0,
\end{align}
where
\begin{align*}
&\tilde a^{ij}(t,x,z) = a^{ij}( t, x,\delta z),~
\tilde b^i(t,x,z) = b^i(t, x,\delta z),~
\tilde \tilde{\bar b}^i (t,x,z) = \bar b^i( t, x,\delta z) \\
& \tilde c(t,x,z)= c(t,x,\delta z),~
\tilde f^i (t,x) = \frac{1}{\delta} f^i(t, x),~
\tilde f (t,x) = \frac{1}{\delta} f( t,  x),~
\tilde g^k (t,x) = \frac{1}{\delta} g^k(t, x).
\end{align*}
By taking $\delta$ small enough, \eqref{0210 eqn v} has a solution $v \in \frH^1_{p,d}(\cO,\tau')$ due to Step 1,
and therefore \eqref{nu zero} has a solution $u \in \frH^1_{p,d}(\cO,\tau')$.

\end{proof}

\mysection{Proof of Theorem \ref{main thm}}

First we prove the uniqueness result.

\begin{lemma}
                                            \label{uni lem}
Suppose that Assumptions \ref{as mea}, \ref{as el}, \ref{as bd 2}, and \ref{as co con} hold.   
Let $f^{i}\in \bH^{\gamma_0}_{p,d}(\cO,\tau)$,
$f\in \psi^{-1} \bH^{\gamma_0-1}_{p,d}(\cO,\tau)$,
$g \in\bH^{\gamma_0}_{p,d}(\cO,\tau,\l_2)$, 
and $u_0 \in U^{\gamma_0+1}_{p,d}(\cO)$.
Assume that  $u,v  \in \frH_{2,d}^1(\cO,\tau)$ are solutions to
the  equation
\begin{align*}
du 
&= \Big[D_i \Big[a^{ij}(u)u_{x^j}+b^i(u) u + f^i\Big] + \bar{b}^i(u) u_{x^i} +c(u)u +f\Big]dt+g^k dW^k_t, \quad t\leq \tau \\
&~  \quad u(0,\cdot)=u_0.
\end{align*}
Then $u=v$ in $\frH^1_{2,d}(\cO,\tau)$, and  moreover
$u \in \frH_{p,d,\loc}^{1+\gamma}(\cO,\tau) $ for any $\gamma<\gamma_0$.
\end{lemma}

\begin{proof}
By Lemma \ref{wei lem}(ii),
\begin{align*}
\bE |u_0|_{C^{1-d/p -2/p}(\cO)} 
\leq N\bE\|\psi^{2/p-1}u_0\|^p_{H^{1-2/p}_{p,\theta}(\cO)} 
\leq N\|u_0\|^p_{U^{\gamma_0}_{p,\theta}(\cO)}.
\end{align*}
Thus by Theorem \ref{hol thm}, there exists a $\alpha \in (0,1)$ so that 
$$
\bE |u|^p_{C^{\alpha}([0,\tau]\times \cO)}+\bE |v|^p_{C^{\alpha}([0,\tau]\times \cO)}<\infty.
$$
Define
\begin{align*}
\tau^{(1)}_n  
:= \inf \{ t \leq \tau : 
&|a^{ij}(t,x,u)|_{C^\alpha([0,t]\times \cO)}+|b(t,x,u)|_{C^\alpha([0,t]\times \cO)}  > n \}, 
\end{align*}
\begin{align*}
\tau^{(2)}_n  
:= \inf \{ t \leq \tau : 
&|a^{ij}(t,x,v)|_{C^\alpha([0,t]\times \cO)}+|b(t,x,v)|_{C^\alpha([0,t]\times \cO)}> n \}, 
\end{align*}
and $\tau_n=\tau^{(1)}_n \wedge \tau^{(2)}_n$.  
Then $\tau_n \to \tau$ $(a.s.)$ as $n \to \infty$ and
by Theorem \ref{lin thm}  for each $n$ we have $u,v\in \frH^{1+\gamma}_{p,d}(\tau_n,\cO)$ for all $\gamma < \alpha$.
Fix $0<\gamma <\alpha $.
Due to \eqref{embed eqn},
\begin{align*}
&\bE | u |^p_{C([0,\tau_n] \times \cO)}
+\bE | \psi^\gamma Du|^p_{C([0,\tau_n]\times \cO)} 
+\bE | v |^p_{C([0,\tau_n] \times \cO)}
+\bE | \psi^\gamma Dv|^p_{C([0,\tau_n]\times \cO)} 
< \infty.
\end{align*}
Thus there exists a sequence of stopping times $\tau_{n,m} \leq \tau_n$ such that
$\tau_{n,m}$ converges to $\tau_n$ $(a.s.)$ as $m \to \infty$, and (a.s.)
\begin{equation}
	                   \label{0210 eqn 1}				
| u |_{C([0,\tau_n] \times \cO)}^p +| v |_{C([0,\tau_n] \times \cO)}^p
+|\psi^\gamma Du |_{C([0,\tau_n] \times \cO)}^p 
+|\psi^\gamma Dv |_{C([0,\tau_n] \times \cO)}^p \leq m.
\end{equation}

\noindent
Due to Assumption \ref{as co con},
\begin{align*}
&D_i\left(a^{ij }(u)u_{x^j}-a^{ij}(v)v_{x^j}\right)\\
&=D_i \left(\int_0^1 \frac{d}{d \vartheta} \left(a^{ij}(\vartheta u + (1-\vartheta)v)(\vartheta u  - (1-\vartheta)v)_{x^j} \right) d\vartheta \right) \\
&= D_i \left(  \int_0^1 a^{ij }(\vartheta u + (1-\vartheta)v)\,d \vartheta \,  (u -v)_{x^j} \right)\\
&\quad+D_i\left(\int_0^1  a_u^{ij }(\vartheta u+ (1-\vartheta)v)  (\vartheta u  - (1-\vartheta)v)_{x^j}d\vartheta \,\,(u - v) \right) 
.
\end{align*}
Similarly,
\begin{align*}
&D_i\left(b(u)u-b^i(v)v \right)\\
&=D_i \left(\int_0^1 \frac{d}{d \vartheta} \left(b(\vartheta u + (1-\vartheta)v)(\vartheta u  - (1-\vartheta)v) \right) d\vartheta \right)\\
&= D_i \left(  \int_0^1 b(\vartheta u + (1-\vartheta)v)\,d \vartheta \,  (u -v) \right)\\
&\quad+ D_i\left(\int_0^1  b_u(\vartheta u+ (1-\vartheta)v)  (\vartheta u  - (1-\vartheta)v) d\vartheta \,\,(u - v) \right) 
,
\end{align*}
\begin{align*}
&\bar b^i(u)u_{x^i}-\bar b^i(v)v_{x^i} \\
&= \int_0^1 \frac{d}{d \vartheta} \left(\bar b^i(\vartheta u + (1-\vartheta)v)(\vartheta u  - (1-\vartheta)v)_{x^i} \right) d\vartheta \\
&=  \int_0^1 \bar b^i(\vartheta u + (1-\vartheta)v)\,d \vartheta \,  (u -v)_{x^i} \\
&\quad+ \int_0^1  \bar b^i_u(\vartheta u+ (1-\vartheta)v)  (\vartheta u  - (1-\vartheta)v)_{x^i} d\vartheta \,\,(u - v)  
,
\end{align*}
and
\begin{align*}
&c(u)u-c(v)v \\
&= \int_0^1 \frac{d}{d \vartheta} \left(c(\vartheta u + (1-\vartheta)v)(\vartheta u  - (1-\vartheta)v) \right) d\vartheta \\
&= \int_0^1 c(\vartheta u^1 + (1-\vartheta)v)\,d \vartheta \,  (u -v) \\
&\quad+ \int_0^1  c_u(\vartheta u+ (1-\vartheta)v)  (\vartheta u  - (1-\vartheta)v) d\vartheta \,\,(u - v) 
.
\end{align*}
Defining 
\begin{align*}
&\tilde a^{ij} :=\int_0^1 a^{ij }(\vartheta u + (1-\vartheta)v)\,d \vartheta \\
&\tilde b := \int_0^1  a_u^{ij }(\vartheta u+ (1-\vartheta)v)  (\vartheta u  - (1-\vartheta)v)_{x^j}d\vartheta
+\int_0^1 b(\vartheta u + (1-\vartheta)v)\,d \vartheta \\
&\quad + \int_0^1  b_u(\vartheta u+ (1-\vartheta)v)  (\vartheta u  - (1-\vartheta)v) d\vartheta \\
&\tilde{\bar b}^i := \int_0^1 \bar b^i(\vartheta u + (1-\vartheta)v)\,d \vartheta  \\
&\tilde c := \int_0^1  \bar b^i_u(\vartheta u+ (1-\vartheta)v)  (\vartheta u  - (1-\vartheta)v)_{x^i} d\vartheta
+\int_0^1 c(\vartheta u^1 + (1-\vartheta)v)\,d \vartheta \\
&\quad +\int_0^1  c_u(\vartheta u+ (1-\vartheta)v)  (\vartheta u  - (1-\vartheta)v) d\vartheta,
\end{align*}
we have
\begin{align*}
(u-v)_t
=: D_i[\tilde{a}^{ij}(u-v)_{x^j}+\tilde{b}^i(u-v)] + \tilde{\bar b}^i(u-v) + \tilde c (u-v), \quad t \leq \tau_{n,m}.
\end{align*}
Due to \eqref{0210 eqn 1}, \eqref{0213 eqn 1}, and the definition of stopping times $\tau_n^{(1)}, \tau_n^{(2)}$, 
the coefficients $\tilde a^{ij}$ are uniformly continuous and 
$$
\sup_{\omega}\sup_{t\leq \tau_{n,m}} 
\left(\left|\rho \left(\tilde{b}+\tilde{\bar b}^i\right)\right| +\left|\rho^2 c\right| \right)
\to 0
$$ 
as $\rho(x) \to 0$.
Thus we can apply Theorem \ref{lin thm} to conclude $u=v$ 
as an element in $\frH_{p,d}^1(\cO,\tau_{n,m})$.
To obtain better regularity,  recall Assumption \ref{as bd 2}.
Due to \eqref{0210 eqn 1}, there exists a $\bar K$ such that for any $\gamma < \gamma_0 $ and $(\omega,t)$,
\begin{align*}
|a^{ij}(t,\cdot,u(\cdot))|^{(0)}_{\gamma_+}+|\psi b^i(t,\cdot,u(\cdot))|^{(0)}_{\gamma_+} 
+|\psi \bar b^i (t,\cdot,u(\cdot))|^{(0)}_{\gamma_+}+|\psi^2 c(t,\cdot,u(\cdot))|^{(0)}_{\gamma_+} 
\leq \bar K.
\end{align*}
Therefore by Theorem \ref{lin thm} again, we conclude
$u \in \frH_{p,d}^{1+\gamma}(\cO,\tau_{n,m})$. 
The lemma is proved.
\end{proof}

%

\vspace{3mm}
{\bf Proof of Theorem \ref{main thm}}
\vspace{3mm}

We  prove that  there is a unique  solution $u$ 
in the class $\frH^1_{2,d}(\cO,\tau)$ and also show this solution belongs to $\frH^{1+\gamma}_{p,d,\loc}(\cO,\tau)$ for any $\gamma<\gamma_0$. The estimates \eqref{main est} and \eqref{holder main}
 are easy  consqeuenes of Theorem \ref{l2 thm} and Corollary \ref{hol cor}.

We divide our proof into two steps.

\vspace{2mm}
{\bf Step 1.} We assume $\nu(t,x)=0$.
\vspace{2mm}

Due to Lemma \ref{uni lem}, we only need to prove the existence of a solution in the class $\frH^1_{2,d}(\cO,\tau)$. 
Define 
\begin{align*}
\Pi 
:= \{\text{stopping times $\tau_{a} \leq \tau$: equation (\ref{main eqn}) has a  solution}~
u\in  \frH^1_{2,d}(\cO,\tau_a)\, \}.
\end{align*}
Observe that if $\tau_{{a_1}}, \tau_{{a_2}} \in \Pi$, then $\tau_{a_1} \vee \tau_{a_2} \in \Pi$. 
Indeed, if $u^1, u^2$ are solutions for $t\leq \tau_{a_1}$ and $t\leq \tau_{a_2}$, respectively, then $u^1=u^2$ for $t\leq \tau_{a_1} \wedge \tau_{a_2}$ by Lemma \ref{uni lem}.  
Thus one can find a solution  $u\in  \frH^1_{2,d}(\cO,\tau_a)$ 
defining $u=u^1$ if $\tau_{a_1}\geq \tau_{a_2}$, and $u=u^2$ otherwise.

 Define $r:= \sup_{\tau_a \in \Pi} \{\bE \tau_a\}$. Then there exist nondecreasing stopping times $\tau_n$  (otherwise one can consider $\tau_1 \vee \cdots \vee \tau_n$) such that
 $\bE\tau_n \uparrow r$.  Put $\bar{\tau} := \lim_{n \to \infty} \tau_n$. Then $\bar{\tau}$ becomes a stopping time
since the filtration $\cF_t$ is right continuous
and $\bE\bar{\tau} = r$ by the monotone convergence theorem.

We will show $\bar{\tau} \in \Pi$.  Since $\tau_n$ are nondecreasing, using the uniqueness result of Lemma \ref{uni lem} we conclude that there exists $u \in \frH^{1}_{2, d, \loc}(\bar \tau,\cO)$ such that $u(0,\cdot)=u_0$ and 
for each $n$
\begin{equation}
					\label{0202 eqn 1}
  du=\left( D_i \left[a^{ij}(u)u_{x^j}+b(u)u+f^i\right]+\bar b(u)u_{x^i}+c(u)u+f\right)+g^kdW^k_t,\quad t\leq \tau_n.
\end{equation}
Moreover due to Theorem \ref{l2 thm} ,
\begin{align*}
\|u\|_{\frH_{2,d}^1(\cO,\tau_n)} 
&\leq N \left(\sum_{i=1}^d\|f^i\|_{\bL_{2,d}(\cO,\bar \tau)}
+\|\psi f\|_{\bH^{-1}_{2,d}(\cO, \bar \tau)}+\|g\|_{\bL_{2,d}(\cO,\bar \tau, l_2)}\right)
\end{align*}
where $N$ is independent of $n$. 
Thus taking $n \to \infty$, we have
\begin{equation}
  \label{eqn 3.17.1}
\|\psi^{-1}u\|_{\bH_{2,d}^1(\cO,\bar \tau)} < \infty.
\end{equation}
This implies that the right hand side of \eqref{0202 eqn 1}   is $\psi^{-1}H^{-1}_{2,d}$-valued continuous function on $[0, \bar \tau]$,  and therefore $u$  is continuously extendible to $[0, \bar \tau]$.
Therefore   \eqref{0202 eqn 1} holds for $t\leq \bar{\tau}$. This with  (\ref{eqn 3.17.1})  shows that   $u \in \frH_{2,d}^1(\cO,\bar \tau)$.  Consequently, $\bar \tau \in \Pi$, and by Theorem \ref{embedding},  $u\in C([0,\bar{\tau}];L_{2,d}(\cO))$ (a.s.).

Next we claim $\bar{\tau}=\tau$ $(a.s)$. Suppose not. Then $P(\bar{\tau}<\tau)>0$.   
By Theorem \ref{embedding}(iii), 
$$
\|u(\bar{\tau},\cdot)\|^2_{U^1_{2,d}(\cO)}=\bE \|u(\bar \tau, \cdot)\|^2_{L_{2,d}(\cO)} < \infty.
$$
Denote $\bar u_0 := u(\bar \tau, \cdot)$, $\cF^{\bar{\tau}}_t:=\cF_{t+\bar{\tau}}$, and  $\bar{w}^k_{t}:= w^k_{t+\bar{\tau}} - w^k_{\bar{\tau}}$. Then $\bar{w}^k_t$ are independent Wiener processes relative to  $\cF^{\bar{\tau}}_t$, 
$\bar u_0$ is $\cF_0^{\bar \tau}$-measurable, and 
  $\tilde{\tau}:=\tau-\tilde \tau$ is  a nonzero stopping time with respect to $\cF^{\bar{\tau}}_t$.
  Consider the equation
\begin{align*}
 d\bar{v} 
 &= \Big[D_i\left(a^{ij}(\bar{\tau}+t,x,\bar{v})\bar{v}_{x^j} + b(\bar \tau + t,x, \bar v)v+f^i(\overline{\tau}+t ) \right) \\
 &\qquad +\bar b^i (\bar \tau + t,x,v)v_{x^i} + c(\bar \tau +t,x,v)v+f(\overline{\tau}+t)\Big]dt  \\
 &\quad + g^k(\bar \tau +t) d\bar{w}^k_t,\qquad t \leq \tilde \tau; \quad \bar{v}(0, \cdot) =\bar{u}_0.
\end{align*}
Then  by Lemma \ref{local sol}, there exists a nonzero stopping time $\sigma\leq \tau-\bar \tau$ (with respect to $\cF^{\bar{\tau}}_t$) so that the above equation has an  $\cF^{\bar{\tau}}_t$-adapted solution 
$$
\bar{v}\in \frH^1_{p,d}(\cO,\sigma) \subset \frH^1_{2,d}(\cO,\sigma).
$$
Define $\tau_0=\bar{\tau}+\sigma$. It is easy to check that $\bE \tau_0>r$ and $\tau_0$ is a stopping time since $\cF_t$ is right continuous.
Define
$$
U(t,x) = \begin{cases}
u(t,x)  \quad &: t \leq \bar{\tau}\\
 \bar{v}(t-\bar{\tau} ,x) \quad &:  \bar{\tau}<t  \leq  \tau_0
\end{cases}
$$
then $U$ satisfies (\ref{main eqn}) for   $t\leq \tau_0$ (a.s.) and $U\in \frH^1_{2,d}(\cO,\tau_0)$.   
This is a contradiction since  $\bE \tau_0 >r$ and $\tau_0 \in \Pi$. 
Therefore, we conclude
$\bar{\tau} = \tau$, and the proof for the case $\nu=0$  is completed.

\vspace{2mm}
{\bf Step 2.} (General case)
\vspace{2mm}
Let $\gamma < \gamma_1 < \gamma_0$
and denote
$$h(t,x)=e^{-\int_0^t \nu^k(s,x)dW^k_s}, \quad 
\tau_n =\inf \{ t \leq \tau : |h(t,\cdot)|^{(0)}_{1+\gamma_1} > n\},
$$
$$
\tilde a^{ij}( t, x, z)=a^{ij}(t,x,hz) ,\quad 
\tilde b^i(t, x,  z)=b^i(t,x,h z),\quad
\tilde{\bar b}^i( t, x, z)=\bar b^i (t,x,h z),
$$
$$
\tilde c(t,x,z)=c(t,x,h z) -\sum_k \left(\nu^k(t,x)\right)^2, \quad 
\tilde f^i=h^{-1}f^i,\quad
\tilde f=h^{-1}f,\quad
\tilde g^k:=h^{-1} g^k.
$$
Then Assumptions \ref{as mea}, \ref{as el},  \ref{as bd 2}, and \ref{as co con} hold 
with $\tilde a^{ij}$, $\tilde b^i$, $\tilde{\bar b}^i$, $\tilde c$, $\tilde f^i$, $\tilde f$, and $\tilde g$ 
on $[0,\tau_n]$.
Therefore by Step 1, there exists a 
$u^{(n)} \in \frH^1_{2,d}(\cO,\tau_n) \cap \frH^{1+\gamma}_{p,d,\loc}(\cO,\tau_n)$, $\gamma <\gamma_0$, 
such that
\begin{align}
du^{(n)} 
						\notag
&= \Big[D_i \Big(\tilde a^{ij}(u^{(n)})u^{(n)}_{x^j}+\tilde b^i(u^{(n)}) u^{(n)} + f^i\Big) 
+ \tilde{\bar{b}}^i(u^{(n)}) u^{(n)}_{x^i} + \tilde c(u^{(n)})u^{(n)} +\tilde f\Big]dt, \\
					\label{0216 eqn}
&~ \quad +\tilde g^k dW^k, \qquad  t  \leq \tau_n; \quad u^{(n)}(0,\cdot)=u_0.
\end{align}
Denote 
$$
v^{(n)}(t,x):=u^{(n)}(t,x) e^{-\int_0^t \nu^k(s,x)dW^k_s}=u(t,x)h(t,x).
$$
Then 
$$
v^{(n)} \in \frH^1_{2,d}(\cO,\tau_n) \cap \frH^{1+\gamma}_{p,d,\loc}(\cO,\tau_n).
$$ 
Moreover by It\^o's formula
\begin{align*}
dh
= \left(h\left(\nu^k\right)^2 \right)dt-\left(h\nu^k\right)d{W_t^k}  
\end{align*}
and 
\begin{align}
dv^{(n)} 
						\notag
&= u^{(n)}(dh) + (du^{(n)})h \\
						\notag
&= \Big[D_i \Big(\tilde a^{ij}(u^{(n)})u^{(n)}_{x^j}h+\tilde b^i(u^{(n)}) u^{(n)}h + \tilde f^ih\Big)  \\
						\notag
&\quad+ \tilde{\bar{b}}^i(u^{(n)}) (u^{(n)}h)_{x^i} 
+\tilde c(u^{(n)})u^{(n)}h +\tilde fh \Big]dt \\
						\notag
& \quad+\left(u^{(n)}h\left(\nu^k\right)^2 \right)dt-\left(u^{(n)}h\nu^k+\tilde g^kh\right)dW_t^k \\
						\notag
&= \left[D_i \Big(a^{ij}(v^{(n)})v^{(n)}_{x^j}+b^i(v^{(n)}) v^{(n)} + f^ih\Big) + \bar{b}^i(v^{(n)}) v^{(n)}_{x^i} +c(v^{(n)})v^{(n)} +f\right] dt \\
						\label{0216 eqn 1}
& \quad+\left( \nu^k v^{(n)}+ g^k\right)dW_t^k,
\qquad t \leq \tau_n; \quad u(0,\cdot)=u_0.
\end{align}
Using the uniqueness result of  equation \eqref{0216 eqn} (Lemma \ref{uni lem}), we conclude that $v^{(n)}$ is also unique solution to  (\ref{0216 eqn 1}), and
\begin{align*}
v^{(n)} = v^{(m)} \quad \text{on} \quad [0,\tau_n) \quad \forall n \leq m.
\end{align*}
Therefore there exists a $v \in \frH^1_{2,d,\loc}(\cO,\tau) \cap \frH^{1+\gamma}_{p,d,\loc}(\cO,\tau)$ 
such that
\begin{align*}
v^{(n)} = v\quad \text{on} \quad [0,\tau_n) \quad \forall n.
\end{align*}

By Theorem \ref{l2 thm}, there exists a $\tilde v \in \frH^1_{2,d}(\cO,\tau)$ to the equation
\begin{align}
d\tilde v
						\notag
&= \left[D_i \Big( a^{ij}(v)\tilde v_{x^j}+b^i(v) \tilde v + f^ih\Big) + \bar{b}^i(v) \tilde v_{x^i} +c(v)\tilde v +f\right] dt \\
						\label{0216 eqn 1}
& \quad+\left( \nu^k \tilde v+ g^k\right)dW_t^k,
\qquad t \leq \tau, \quad x \in \cO, \quad u(0,\cdot)=u_0.
\end{align}
Due to Lemma \ref{uni lem}, 
\begin{align*}
v = \tilde v\quad \text{on} \quad [0,\tau_n) \quad \forall n
\end{align*}
and therefore 
$$
v=\tilde v \in \frH^1_{2,d}(\cO,\tau) \cap \frH^{1+\gamma}_{p,d,\loc}(\cO,\tau).
$$
The theorem is proved. \qed

\end{document}